\documentclass[11pt]{amsart}
\usepackage{amsmath}
\usepackage{amsthm}
\usepackage{amssymb}
\usepackage{verbatim}
\usepackage{esint}
\usepackage{color}

\newtheorem{theorem}{Theorem}    
\newtheorem{proposition}[theorem]{Proposition}

\newtheorem{lemma}[theorem]{Lemma}

\newtheorem{remark}[theorem]{Remark}
\newtheorem{definition}[theorem]{Definition}
\theoremstyle{definition}

\numberwithin{theorem}{section} \numberwithin{theorem}{section}
\numberwithin{equation}{section}

\def\C{\mathbb{C}}

\def\F{\mathcal{F}}

\def\supp{\operatorname{supp}}

\allowdisplaybreaks

\begin{document}
\title[Boundedness of commutators for  multilinear Calder\'{o}n-Zygmund  operators ]
{Boundedness of commutators for  multilinear Calder\'{o}n-Zygmund  operators on Generalized Morrey Spaces}

\author{Fuli Ku}


%
\thanks{Supported by the NNSF of China (Nos.12171399, 11871101)}
\address{School of Mathematical Sciences, Xiamen University, Xiamen 361005, China} \email{kfl20180325@163.com}



\begin{abstract}
Let $T$ be a $m$-linear Calder\'{o}n-Zygmund operator of type $\omega$ with $\omega$ being nondecreasing and $\omega \in$ Dini(1) and $[\vec{b},\,T]$ be the commutator generated by $T$ with symbols $\vec{b}=(b_1,\,\ldots,\,b_m)$ belonging to  generalized Campanato spaces. We give necessary and  sufficient conditions for the boundedness of $[\vec{b},\,T]$ on generalized Morrey spaces with variable growth condition.
\end{abstract}

\maketitle

\section{Introduction }
It is well known that the boundedness of operators on function spaces  is a central topic
of harmonic analysis, which attracts a lot of attentions. In this paper, we will focus
on the boundedness of the commutators for the $m$-linear Calder\'{o}n-Zygmund operators of type $\omega$, which are defined as follow.
\begin{definition}\label{def1.1}
A locally integrable function $K(x,y_1,\ldots,y_m)$, defined away from the diagonal $x = y_1 =\ldots =y_m$ in $(\mathbb{R} ^n)^{m+1}$, is called a $m$-linear Calder\'{o}n-Zygmund operator kernel of type $\omega$, with that $\omega:[0,\infty)\rightarrow[0,\infty)$ and satisfies $\int_0^1\frac{\omega(t)}{t}dt<+\infty$, if there exists a constant $A > 0$ such that
\begin{equation}\label{eq1.1}\,\
 \left | K(x,y_1,\dots,y_m) \right | \le \frac{A}{ {(\sum_{i=1}^{m} \mid x-y_i \mid )}^{mn} }
\end{equation}
for all $(x,y_1,\dots,y_m) \in (\mathbb{R} ^n)^{m+1}$ with $x \neq y_i$
for some $i \in \{1, 2, . . . , m\}$, and
\begin{equation}\label{eq1.2}\,\
 | K(x,y_1,\dots,y_m ) - K(x',y_1,\dots,y_m ) |
\le \frac{A}{ {(\sum_{i=1}^{m} | x-y_i| )}^{mn }}\omega \Big( \frac{| x-x'|}{\sum_{i=1}^{m} | x-y_i |}\Big )
\end{equation}
whenever $|x-x'| \leq \frac{1}{2}\max_{1\leq i\leq m} | x - y_i|$, and
\begin{equation}\label{eq1.3}\,\
\begin{split}
 &| K(x,y_1,\ldots,y_i,\ldots,y_m ) - K(x,y_1,\ldots,y_i',\ldots,y_m ) |
 \\&\qquad\le\frac{A}{ {(\sum_{i=1}^{m} | x-y_i | )}^{mn }  }
\omega \Big ( \frac{ | y_i-y_i' | }{\sum_{i=1}^{m} | x-y_i |  }  \Big )
\end{split}
\end{equation}
whenever $| y_i - y_i'|\leq \frac{1}{2} \max_{1\leq j\leq m} | x - y_j|$.

We say $T:\mathcal{S}(\mathbb{R} ^n) \times\ldots \times \mathcal{S}(\mathbb{R} ^n) \to \mathcal{S}'(\mathbb{R} ^n)$ is a $m$-linear operator with  kernel $K(x,y_1,\dots,y_m )$, if
\begin{equation}\label{eq1.4}\,\
T(\vec{f} )(x)=T(f_1,\dots,f_m)(x):=\int_{(\mathbb{R} ^n)^{m}}K(x,y_1,\dots,y_m)f_1(y_1)\dots f_m(y_m) dy_1\dots dy_m
\end{equation}
whenever $x \notin \bigcap_{i=1}^{m} \supp f_i$ and  $f_i\in C_c^\infty(\mathbb{R} ^n),i=1,\ldots,m$.

If $T$ can be extended to a bounded $m$-linear operator from $L^{p_1}(\mathbb{R}^n)\times\cdots \times L^{p_m}(\mathbb{R} ^n)$ to $L^{p,\infty}(\mathbb{R} ^n)$ for some $p_i,\,p\in(1,\infty)$  with $\sum_{i=1}^{m} 1/{p_i} =1/p$, or, from $L^{p_1}(\mathbb{R} ^n)\times\cdots \times L^{p_m}(\mathbb{R} ^n)$ to $L^{1}(\mathbb{R} ^n)$, for some $p_i\in(1,\infty),i=1,\ldots,m$ with $\sum_{i=1}^{m} 1/{p_i} =1$, then $T$ is called a $m$-linear Calder\'{o}n-Zygmund operator of type $\omega$, abbreviated to $m$-linear $\omega$-$\mathrm{CZO}$.
\end{definition}
For notational convenience, we will occasionally write
$$\vec{y}:=(y_1,\dots,y_m ),\,K(x,\vec{y}):=K(x,y_1,\dots,y_m ),\, d\vec{y}:=dy_1\ldots dy_m.$$

\begin{definition}
Given a collection of locally integrable functions $~\vec{b} =(b_1,\dots,b_m)$. If $T$ is the $m$-linear
Calder\'{o}n-Zygmund  operator, then the $m$-linear commutators of $T$ are defined by
\begin{equation}\label{eq1.5}\,\
[\vec{b},T](\vec{f})(x):=\sum_{j=1}^{m}T_{\vec{b}}^j(\vec{f})(x),
\end{equation}
where each term is the commutator of $T$ with $b_j$ in the $j$-th entry, that is,
$$T_{\vec{b}}^j(\vec{f})(x):=b_j(x)T(f_1,\dots,f_j,\dots,f_m)(x)-T(f_1,\dots,b_jf_j,\dots,f_m)(x).$$
\end{definition}
The $m$-linear commutators were first considered by P\'{e}rez  and Torres \cite{CPRH}. Later on, Lerner etal.\cite{LAK} introduced the multiple weights $ A_{\vec{p}}$  and proved that for $\vec{b}\in (\mathrm{BMO})^m$, $[\vec{b},T]$ is bounded from  $L^{p_1}(\omega_1)\times\ldots \times L^{p_m}(\omega_m)$
to $L^{p}(\vec{\omega})$ for $\vec{\omega}=(\omega_1,\ldots,\omega_m)\in A_{\vec{p}}$, the multiple Muckenhoupt  classes. A pilar for such considerations in bilinear setting is the work of Ding and Mei \cite{DM}, where they showed that the boundedness of bilinear Calder\'{o}n-Zygmund commutators on Morrey space.
 Xue and Yan \cite{XY} showed the boundedness of generalized commutators of multilinear Calder\'{o}n-Zygmund type
operators. Moreover, Chaffee \cite{LC} given the boundedness  of the bilinear singular integral operator commutator to characterize BMO.
Recently, Kunwar and Ou \cite{KO} and Li \cite{LK} obtained the Bloom type multiple weight inequalities of $[\vec{b},T]$. Guo and Wu \cite{GLW} obtained the unified theory for the necessity of bounded commutators, then continued by many authors (see \cite{LGRHT,MLCT,CSWH,CPRT,LZ,PT,WZ,LOR,DWZT,LYRY,IAF} etc.).

On the other hand,  for $~m=1$, Arai and Nakai \cite{ARNE1} recently studied the boundedness for commutators $[b,T]$ of Calder\'on-Zygmund operator $T$ on the generalized Morrey spaces. They showed that if $b$ belongs to generalized Campanato spaces $\mathcal{L}^{(1,\psi)}(\mathbb{R}^n)$, then $[b,T]$ is bounded on the generalized Morrey spaces. The corresponding result for the commutators of general fractional integrals is also obtained.

Based on the previous results mentioned above, we will consider the boundedness of $m$-linear commutators $[b,T]$ on the generalized Morrey spaces. In
addition. We will give necessary and  sufficient conditions for the boundedness of the commutator $[\vec{b},\,T]$ on generalized Morrey spaces with variable growth condition.
To state our main results, we first recall some relevant definitions and notation.

Let $B(x,r)$ be the open ball of radius $r$ centered at $x\in \mathbb{R}^n$, that is,
  $$
  B(x,r)=\{y\in \mathbb{R}^n:|y-x|<r\}.
  $$

For a measurable set $E\subset \mathbb{R}^n$, we denote by $\left | E \right | $ and $\chi _E$ the
Lebesgue measure of $E$ and the characteristic function of $E$, respectively. For a function $f\in L_{\mathrm{loc}}^{1} (\mathbb{R}^n)$ and a
ball $B$, let
$$f_B=\fint_{B}f(y)dy=\frac{1}{\left | B \right | }\int_{B}f(y)dy.
$$
Moreover, we denote by $\varphi(B)=\varphi(x,r)$, for a measurable function $\varphi:\mathbb{R}^n\times(0,\infty)\rightarrow(0,\infty)$, while a ball $B = B(x,r)$.

\begin{definition}[\cite{ARNE1}]
Let $\varphi(x,r)$ be a positive measurable function on
$\mathbb{R}^n\times (0,\infty )$ and $p\in[1,\infty)$, the generalized Morrey space $L^{(p,\varphi)}(\mathbb{R}^n)$ is denoted  by
$$
L^{(p,\varphi)}(\mathbb{R}^n):=\Big\{f: \|f\|_{L^{(p,\varphi)}(\mathbb{R}^n)}=\sup_{B}\Big( \frac{1}{\varphi(B)}\fint_{B} |f (y)|^p dy\Big)^{1/p}<\infty \Big \},$$
where the supremum is taken over all balls $B$ in $\mathbb{R}^n$.

We recall that $\|f\|_{L^{(p,\varphi)}(\mathbb{R}^n)}$ is a norm and $L^{(p,\varphi)}(\mathbb{R}^n)$ is a Banach space.
If $\varphi _{\lambda } (x,r)=r^{\lambda }$ for $\lambda\in[-n,0]$, then ${L^{(p,\varphi)}(\mathbb{R}^n)}$ is the classical Morrey space, that is,
$$
\|f\|_{L^{(p,\varphi_{\lambda})}(\mathbb{R}^n)}
=\sup _{B}\Big(\frac{1}{\varphi_{\lambda}(B) } \fint_{B}|f(y)|^{p} d y\Big)^{1 / p}
=\sup _{B=B(x, r)}\Big(\frac{1}{r^{\lambda}} \fint_{B}|f(y)|^{p} d y\Big)^{1 / p}.
$$
\end{definition}
In particular, ${L^{(p,\varphi_{-n})}(\mathbb{R}^n)}={L^{p}(\mathbb{R}^n)}$, and ${L^{(p,\varphi_{0})}(\mathbb{R}^n)}={L^{\infty}(\mathbb{R}^n)}$.

\begin{definition}[\cite{ARNE1}]
Let $\psi(x,r)$ be a positive measurable function on
$\mathbb{R}^n\times (0,\infty )$ and $p\in[1,\infty)$,
the generalized Campanato spaces $\mathcal{L}^{(p,\psi)}(\mathbb{R}^n)$ is defined  by
$$
\mathcal{L}^{(p, \psi)}(\mathbb{R}^{n}):=\Big \{f\in L_{\mathrm{loc}}^1(\mathbb{R}^n): \|f\|_{\mathcal{L}^{(p, \psi)}(\mathbb{R}^n)}<\infty \Big \},
$$
where $ \|f\|_{\mathcal{L}^{(p, \psi)}(\mathbb{R}^n)}=\sup _{B}\Big(\frac{1}{\psi(B) } \fint_{B}|f(y)-f_{B}|^{p} d y\Big)^{1/p}$, the supremum is taken over all balls $B$ in $\mathbb{R}^{n}$.
\end{definition}
It is easy to check that $\|f\|_{\mathcal{L}^{(p, \varphi)}(\mathbb{R}^{n})}$ is a norm modulo constants and $\mathcal{L}^{(p, \varphi)}(\mathbb{R}^{n})$ is a Banach space. If $p=1$ and $\varphi\equiv1$, then $\mathcal{L}^{(p, \varphi)}(\mathbb{R}^{n})=\mathrm{BMO}(\mathbb{R}^{n})$. If $p=1$ and
$\varphi(x,r)=r^\alpha\,(0<\alpha\leq1)$, then $\mathcal{L}^{(p, \varphi)}(\mathbb{R}^{n})$ coincides with ${\rm Lip}_{\alpha}(\mathbb{R}^{n})$.

For $f_i\in L^{(p_i,\varphi_i)}(\mathbb{R}^n),\,1<p_i<\infty$, for each ball $B\subset\mathbb{R}^n$, let $f_i^0:=f_i\chi_{2B},\,f_i^{\infty}:=f_i\chi_{(2B)^\complement},\,i=1,\ldots, m$. Here, and in what follows, $E^\complement=\mathbb{R}^n\backslash E$ denotes the complementary set of any measurable subset $E$ of $\mathbb{R}^n$. Then
\begin{equation}\label{eq1.6}\,\
\begin{split}
\prod_{i=1}^{m} f_{i} &=\prod_{i=1}^{m}\left(f_{i}^{0}+f_{i}^{\infty}\right) =\sum_{\alpha_{1}, \ldots, \alpha_{m} \in\{0, \infty\}} f_{1}^{\alpha_{1}} \ldots f_{m}^{\alpha_{m}}
\\&=: \vec{f^0}+ \vec{f^\infty }+\widetilde{\sum} f_{1}^{\alpha_{1}} \ldots f_{j}^{\alpha_{j}} \ldots f_{m}^{\alpha_{m}},
\end{split}
\end{equation}
where each term of $\widetilde{\sum}$ contains at least one $\alpha_j\neq 0$ or $\infty$  at the same time. We defined
\begin{equation}\label{eq1.7}\,\
T(\vec{f})(x)
:=T(\vec{f^{0}})(x)+T(\vec{f^{\infty}})(x)+\widetilde{\sum } T(f_1^{\alpha_1}\ldots f_m^{\alpha_m})(x),\,\forall x\in B.
\end{equation}

Note that $T(\vec{f^{0}}) $ is well defined since $f_i\chi_{2B}\in L^{p_i}(\mathbb{R}^n),\,i=1,\ldots m$, and it is easy to check that
$$T(\vec{f^{\infty}})(x),\,\widetilde{\sum }T\Big(\prod_{i=1}^{m}f_{i}^{\alpha_{i}}\Big)(x)\le \int_{2r}^{\infty }\frac{\varphi (x,t)^{1/p} }{t}dt\prod_{i=1}^{m}  \| f_i  \|_{L^{(p_i,\varphi _i)}},$$
which converges absolutely. Moreover, $T(\vec{f})(x)$ defined in (\ref{eq1.7}) is independent of the choice of the ball containing $x$. Furthermore, we can show that $T$ is bounded form $L^{(p_1,\varphi_1)}(\mathbb{R}^n)\times\ldots \times L^{(p_m,\varphi_m)}(\mathbb{R}^n)$
to $L^{(p,\varphi)}(\mathbb{R}^n)$. See Lemma \ref{lm3.1} for the details.

Let $f_i\in L^{(p_i,\varphi_i)}(\mathbb{R}^n),1<p_i<\infty$, $\,i=1,\ldots, m$. Employing the notation as in (1.6), we define $T_{\vec{b}}^j(\vec{f})$ on each ball $B$ by
\begin{equation}\label{eq1.8}\,\
T_{\vec{b}}^j(\vec{f})(x):=[b_j,T](\vec{f^0})(x)+[b_j,T](\vec{f^\infty})(x)+\widetilde{\sum } [b_j,T](f_1^{\alpha_1}\ldots f_m^{\alpha_m})(x),
\end{equation}
which is well-definedness, see Remark $\ref{rm3.3}$.

We say that a function $\theta:\,\mathbb{R}^n\times(0,\infty)\rightarrow(0,\infty)$ satisfies the doubling condition if there exists a positive constant $C$ such that, for all $x\in\mathbb{R}^n$ and $r,\,s\in(0,\infty)$,
\begin{equation}\label{eq1.9}
\frac{1}{C} \le \frac{\theta(x,r) }{\theta(x,s)} \le C,\,{\rm if}\,\frac{1}{2} \le \frac{r}{s} \le 2.
\end{equation}

We also consider the following condition that there exists a positive constant $C$ such that, for all $x,\,y\in\mathbb{R}^n$ and $r\in(0,\infty)$,
\begin{equation}\label{eq1.10}
\frac{1}{C} \le \frac{\theta(x,r) }{\theta(y,r)} \le C,\,{\rm if}\,|x-y|\leq r.
\end{equation}

For two functions $\theta,\,\kappa:\,\mathbb{R}^n\times(0,\infty)\rightarrow(0,\infty) $, we denote $\theta\thicksim\kappa$ if there exists a
positive constant $C$ such that, for all $x\in\mathbb{R}^n$ and $r\in(0,\infty)$,
\begin{equation}\label{eq1.11}
\frac{1}{C} \le \frac{\theta(x,r) }{\kappa(x,r)} \le C.
\end{equation}

\begin{definition}
$(i)$ Let $\mathcal{G}^{dec}$ be the set of all functions $\varphi: \mathbb{R}^n\times(0,\infty)\rightarrow(0,\infty)$ such
that $\varphi$ is almost decreasing and that $r\mapsto  \varphi (x,r)r^n$ is almost increasing. That is,
there exists a positive constant $C$ such that, for all $x\in\mathbb{R}^n$ and $r,\,s\in(0,\infty)$,
$$
C \varphi(x, r) \geq \varphi(x, s),\, \varphi(x, r) r^{n} \leq C \varphi(x, s) s^{n},\,\text {if } r<s.
$$

$(ii)$ Let $\mathcal{G}^{inc}$ be the set of all functions $\varphi:\mathbb{R}^n\times(0,\infty)\rightarrow(0,\infty)$ such that $\varphi$ is
almost increasing and that $r\mapsto \varphi(x, r)/r$ is almost decreasing. That is, there
exists a positive constant $C$ such that, for all $x\in\mathbb{R}^n$ and $r,\,s\in(0,\infty)$,
$$
\varphi(x, r) \leq C \varphi(x, s),\, C \varphi(x, r) / r \geq \varphi(x, s) / s,\,\text { if }\,r<s.
$$
\end{definition}
\begin{remark}\label{rm1.6}\,\
{\bf{(i)}}If $\varphi\in\mathcal{G}^{dec}$ or $\varphi\in\mathcal{G}^{inc}$, then $\varphi$ satisfies the doubling condition (\ref{eq1.9}).

{ \bf{(ii)}}It follows from \cite{ARNE1} that, for $\varphi\in \mathcal{G}^{dec}$, if $\varphi$ satisfies
\begin{equation}\label{eq1.12}
\lim_{r \to 0} \varphi (x,r)=\infty,\,\lim_{r \to \infty } \varphi (x,r)=0,
\end{equation}
then there exists $\widetilde{\varphi } \in \mathcal{G}^{dec}$ such that $\varphi\sim\widetilde{\varphi }$ and that $\widetilde{\varphi}(x,\cdot)$ is continuous, strictly decreasing and bijective from $(0,\infty)$ to itself for each $x$.
\end{remark}
Now we can formulate our main result as follows.
\begin{theorem}\label{the1.6}\,\
Let $T$ be a $m$-linear Calder\'{o}n-Zygmund  operator of type $\omega$ with satisfying $\int_{0}^{1} \frac{\omega (t)\log{\frac{1}{t} } }{t} dt<\infty $. Let $1 < p,\,p_i  < \infty,\,i=1,\ldots, m,\,p\leq q $ with $\sum_{i=1}^{m} 1/{p_i} =1/ p$, $\varphi,\,\varphi_i,\,\psi:\,\mathbb{R}^n\times (0,\infty )\rightarrow (0,\infty )$ and satisfy
\begin{equation}\label{eq1.13}
\prod_{i=1}^{m} \varphi _i^{1/{p_i}}=\varphi^{1/{p}}.
\end{equation}
$(i)$Assume that $\psi\in \mathcal{G}^{inc}$ satisfies $(\ref{eq1.10})$, $\varphi,\,\varphi_i\in \mathcal{G}^{dec}$ satisfies $(\ref{eq1.12})$. For all $x\in \mathbb{R}^{n}$ and $r\in (0,\infty)$, there exists a positive constant $C_0,\,C$, such that
\begin{equation}\label{eq1.14}\,\
\psi (x,r)\varphi (x,r)^{1/p }\le C_0\varphi (x,r)^{1/q },
\end{equation}
\begin{equation}\label{eq1.15}
 \int_r^\infty \frac{\varphi(x,t)}tdt\le C\varphi(x,r).
 \end{equation}
If $b_i\in {\mathcal{L}^{(1, \psi)}}(\mathbb{R}^{n})$, then $[\vec{b} ,T](\vec{f} ) $ in $(\ref{eq1.8})$ is well defined for all $f_i \in L^{(p_i,\varphi_i) }(\mathbb{R}^{n})$ and there exists a positive constant $C$, independent of $b_i$ and $f_i$, such that
$$\| [\vec{b} ,T] (\vec{f} ) \| _{L^{(q,\varphi) }} \le C\|\vec{b} \|_{(\mathcal{L}^{(1, \psi)})^m}\prod_{i=1}^{m} \| f_i\| _{L^{(p_i,\varphi_i) }}$$
where $\|\vec{b} \|_{(\mathcal{L}^{(1, \psi)})^m}:=\sup _{j=1,\dots ,m}\| b_j \|_{\mathcal{L}^{(1, \psi)}(\mathbb{R}^n)}$.

$(ii)$Conversely, assume that $\varphi,\,\varphi_i\in \mathcal{G}^{dec}$ satisfies $(\ref{eq1.10})$ and that there exists a positive constant $C_0$, such that, for all
$x\in \mathbb{R}^{n}$ and $r\in (0,\infty)$,
\begin{equation}\label{eq1.16}\,\
C_0\psi (x,r)\varphi (x,r)^{1/p }\geq \varphi (x,r)^{1/q }.
\end{equation}
If $T$ is a convolution type such that
$$
T(\vec{f} )(x)=\mathrm{p.v.}\int_{(\mathbb{R}^n)^m}K(x-y_1,\dots,x-y_m )\vec{f}d\vec{y}
$$
with nonzero homogeneous kernel $K\in C^\infty (S^{mn-1})$ satisfying $K(x)=| x|^{-n}K(x /| x|)$,\\$\,\int_{(S^{mn-1})}Kd\sigma(x')=0 $, and if $[\vec{b},T]$ is bounded from $L^{(p_1,\varphi_1)}(\mathbb{R}^n)\times\ldots \times L^{(p_m,\varphi_m)}(\mathbb{R}^n)$
to $L^{(q,\varphi)}(\mathbb{R}^n)$, then $b_j\in \mathcal{L}^{(1,\psi)}(\mathbb{R}^n),\,j=1,\ldots, m$ and there exists a positive constant $C$, independent of $b_j$, such that
$$\| b_j  \| _{\mathcal{L}^{(1,\psi)}}\le C \| [ b_j ,T] \|_{L^{(p_1,\varphi_1)}\times\ldots\times  L^{(p_m,\varphi_m)}\to L^{(q,\varphi)}}$$
where $~ \| [ b_j ,T] \|_{L^{(p_1,\varphi_1)}\times\ldots\times  L^{(p_m,\varphi_m)}\to L^{(q,\varphi)}}$ is the operator norm of $[b_j,T]$ form $L^{(p_1,\varphi_1)}(\mathbb{R}^n)$\\$\times\ldots \times L^{(p_m,\varphi_m)}(\mathbb{R}^n)$
to $L^{(q,\varphi)}(\mathbb{R}^n)$.
\end{theorem}

 We organize the rest of the paper as follows.
In Section \ref{sec2}, we will recall and establish some auxiliary lemmas. In Section \ref{sec3}, we establish some lemmas and  give the proofs of the boundedness of the generalized $m$-linear maximal operator. Section \ref{sec4}, we will establish the pointwise estimate for the sharp maximal operator of $[\vec{b},\,T]$. The proof of Theorem \ref{the1.6} will be given in Section \ref{sec5}.

Finally, we make some conventions for notations. Throughout this paper, we always use $C$ to denote a positive constant that is independent of the main parameters involved but whose value may differ from line to line. Constants with subscripts, such as $ C_p$, are dependent on the correspending subscripts. We denote $f\lesssim g$ if $f\leq Cg$, and $f\thicksim g$ if $f\lesssim g \lesssim f$. For $1\leq p\leq \infty,$ $p'$ denote the conjugate index of $p$ with $1/p+1/p'=1$.

\section{Auxiliary lemmas\label{sec2}}
In this section, we will recall some previous results and establish some auxiliary lemmas.

\begin{lemma}[\cite{ARNE1}]\label{lem2.1}\,\
Let $p\in(1,\infty)$ and $\psi \in \mathcal{G}^{inc}$. Assume that $\psi$ satisfies $(\ref{eq1.10})$. Then, $\mathcal{L}^{(p,\psi^p)}(\mathbb{R}^n)=\mathcal{L}^{(1,\psi)}(\mathbb{R}^n)$
with equivalent norms.
\end{lemma}

\begin{lemma}[\cite{ARNE1}]\label{lem2.2}\,\
Let $p\in(1,\infty)$ and $\psi\in \mathcal{G}^{inc}$. Assume that $\psi$ satisfies $(\ref{eq1.10})$. Then,
there exists a positive constant $C$ dependent only on $n,\,p$ and $\psi$ such that, for all
$ f\in \mathcal{L}^{(1,\psi)}(\mathbb{R}^n)$ and for all $x\in\mathbb{R}^n$ and $r,\,s\in(0,\infty)$,
\begin{equation}\label{eq2.2}
\left(\fint _{B(x, s)}|f(y)-f_{B(x, r)}|^{p} d y\right)^{1 / p} \leq C \int_{r}^{s} \frac{\psi(x, t)}{t} d t\|f\|_{\mathcal{L}^{(1, \psi)}},\,
{\rm if}\,2 r<s,
\end{equation}
and
\begin{equation}\label{eq2.3}
\left (\fint_{B(x,s)} | f(y)-f_{B(x,r)}| ^pdy \right) ^{1/p}\le C\left(\log_{2}{\frac{s}{r}}\right)\psi (x,s) \|f\| _{\mathcal{L}^{(1,\psi) } },\,
{\rm if}\, 2r<s.
\end{equation}
\end{lemma}

\begin{lemma}[\cite{ARNE1}]\label{lem2.3}
Let $\varphi$ satisfy the doubling condition $(\ref{eq1.9})$ and $(\ref{eq1.15})$, that is,
$$ \int_{r}^{\infty } \frac{\varphi (x,t)}{t} dt\le C\varphi (x,r).$$
Then, for all $p\in (0,\infty)$, there exists a positive constant $C_p$ such that, for all $x\in \mathbb{R}^n$ and $r>0$,
\begin{equation*}
\int_{r}^{\infty } \frac{\varphi (x,t)^{1/p}}{t} dt\le C_p\varphi (x,r)^{1/p}.
\end{equation*}
\end{lemma}

\begin{lemma}[\cite{ARNE1}]\label{lem2.4}\,\
Let $p,\,p_i\in[1,\infty),\,i=1,\ldots,m$ satisfies $\sum_{i=1}^{m} 1/{p_i} =1/p$ and $\varphi,\,\varphi _i:\mathbb{R}^n\times (0,\infty )\rightarrow (0,\infty )$. If $\varphi,\,\varphi_i$ satisfies $(\ref{eq1.13})$, then
$$
\|\prod_{i=1}^{m}f_i  \| _{L^{(p,\varphi )}}\le \prod_{i=1}^{m} \left \| f_i \right \| _{L^{(p_i,\varphi _i)}}.
$$
\end{lemma}

For a function $\rho:\mathbb{R}^n\times (0,\infty )\rightarrow (0,\infty )$, the generalized maximal
fractional operator, which is defined by
$$
M _\rho (f)(x)=\sup_{B\ni x}\rho (B) \fint_{B} \left | f(y) \right |dy .
$$

For the generalized maximal fractional operator $M_\rho$, we have the following lemma.
\begin{lemma}[\cite{ARNE1}]\label{lem2.5}\,\
Let $1<p\leq q<\infty$ and $\rho,\,\varphi$ are positive measurable function on $\mathbb{R}^n\times (0,\infty )$. Assume that $~\varphi$ is in $\mathcal{G}^{dec}$ and satisfies $(\ref{eq1.12})$. Assume also that there exists a positive constant $C_0$ such that,
 for all $x\in \mathbb{R}^n$ and $r\in(0,\infty)$,
\begin{equation}\label{eq2.4}
\rho (x,r)\varphi (x,r)^{1/p}\le C_0\varphi (x,r)^{1/q}.
\end{equation}
Then $M_\rho$ is bounded from $L^{(p,\varphi)}(\mathbb{R}^n)$  to $L^{(q,\varphi)}(\mathbb{R}^n)$. Clearly, if $\rho\equiv1$, then $M _\rho$ is the Hardy-Littlewood maximal operator $M$, we have $ \| M(f) \| _{L^{(p,\varphi)}}\lesssim \| f  \| _{L^{(p,\varphi)}}$.
\end{lemma}

For a function $\rho:\mathbb{R}^n\times (0,\infty )\rightarrow (0,\infty )$, the generalized $m$-linear maximal operator, which is defined by
$$
 \mathcal{M} _\rho (\vec{f})(x)=\sup_{B\ni x}\rho (B)\prod_{i=1}^{m}  \fint_{B} | f_i(y_i) | dy_i.
$$

If $\rho (B)=|B| ^{\alpha /n}$, then $\mathcal{M} _\rho (\vec{f})$ is the usual fraction maximal operator $\mathcal{M} _\alpha (\vec{f})$ defined by
$$
\mathcal{M} _\alpha  (\vec{f})(x)=\sup_{B\ni x}  | B  | ^{\alpha /n}\prod_{i=1}^{m}  \fint_{B} | f_i(y_i) | dy_i.
$$

If $\rho\equiv1$, then $\mathcal{M} _\rho (\vec{f})(x)$ is the $m$-linear  maximal operator $\mathcal{M}$, that is
$$
\mathcal{M} (\vec{f})(x)=\sup_{B\ni x}  \prod_{i=1}^{m}  \fint_{B} | f_i(y_i) | dy_i.
$$

For the boundedness of $\mathcal{M},\,\mathcal{M} _\rho$ are the consequences of the following  lemmas.
\begin{lemma}\label{lem2.6}\,\
 Let $p,\,p_i\in[1 , \infty)$ satisfies $\sum_{i=1}^{m} 1/{p_i} =1/p$ and $\varphi,\,\varphi_i:\mathbb{R}^n\times (0,\infty )\rightarrow (0,\infty )$ satisfies $(\ref{eq1.13})$.
Assume that there exists a positive constant $C$ such that
$$
C\varphi (x,r)\geq\varphi (x,s),\,for \,x\in \mathbb{R}^n,~0<r<s,
$$
then $\mathcal{M}$ is bounded from $L^{(p_1,\varphi_1)}(\mathbb{R}^n)\times\ldots \times L^{(p_m,\varphi_m)}(\mathbb{R}^n)$  to $L^{(p,\varphi)}(\mathbb{R}^n)$.
\end{lemma}

\begin{proof}
Note that $\mathcal{M}(\vec{f})(x)\leq\prod_{i=1}^{m} M(f_i)(x)$, using Lemma $\ref{lem2.4}$ and Lemma $\ref{lem2.5}$, we have
$$
\|\mathcal{M}(\vec{f})\|_{L^{(p,\varphi)}}
 \leq\| \prod_{i=1}^{m} M(f_i)(x)\|_{L^{(p,\varphi)}}
\leq\prod_{i=1}^{m}\| M(f_i)\|_{L^{(p_i,\varphi_i)}}
\lesssim \prod_{i=1}^{m}\| f_i\|_{L^{(p_i,\varphi_i)}}.
$$
\end{proof}

\begin{lemma}\label{lem2.7}\,\
Let $p,\,p_i,\,q\in[1,\infty),\,p<q,\,i=1,\ldots, m$ satisfies $\sum_{i=1}^{m} 1/{p_i} =1/p$. Let $\rho,\,\varphi,\,\varphi_i:\mathbb{R}^n\times (0,\infty )\rightarrow (0,\infty )$. Assume that $\varphi,\,\varphi_i$ is in $\mathcal{G}^{dec}$ and satisfies $(\ref{eq1.9})$, $(\ref{eq1.12})$, $(\ref{eq1.13})$. Assume also that there exists a positive constant $C_0$ such that,
 for all $x\in \mathbb{R}^n$ and $r\in(0,\infty)$,
\begin{equation}\label{eq2.8}
\rho (x,r)\varphi (x,r)^{1/p}\le C_0\varphi (x,r)^{1/q}.
\end{equation}
Then $\mathcal{M}_\rho$ is bounded on $(L^{(p_1,\varphi_1)}(\mathbb{R}^n)\times\ldots \times L^{(p_m,\varphi_m)}(\mathbb{R}^n), L^{(q,\varphi)}(\mathbb{R}^n))$.
\end{lemma}

\begin{proof}
We assume that $\varphi(x,\cdot)$ is continuous, strictly decreasing and bijective from $(0,\infty)$ to itself for each $x\in\mathbb{R} ^n,$ see Remark $\ref{rm1.6}$(ii).
We  consider $f_i\in{L^{({p_i},\varphi_i )}(\mathbb{R} ^n)}$ and
 with $\left \| f_i \right \|_{L^{({p_i},\varphi_i )}(\mathbb{R} ^n)}=1,\,i=1,\ldots, m$. Since Lemma $\ref{lem2.6}$, to obtain Lemma $\ref{lem2.7}$, it suffices to prove for $1<p<q$,
\begin{equation}\label{eq2.9}
\mathcal{M}_\rho (\vec{f} ) (x)\le C\mathcal{M}(\vec{f})(x)^{p/q},\,x\in \mathbb{R}^n,
\end{equation}
for some positive constant $C$ independent of $f_i$ and $x$. To prove (\ref{eq2.9}), we show that for any ball $B = B(x,r)$, we have
$$
\rho (B)\prod_{i=1}^{m}\fint_{B}\left | f_i (y_i)\right | dy_i\le C_0\mathcal{M}(\vec{f})(x)^{p/q}.
$$
Choose $u > 0$ such that $\varphi(x,u) = \mathcal{M}(\vec{f}) (x)^p$. If $r \leq u$, then $\varphi(B) = \varphi(x,r)\geq \mathcal{M} (\vec{f}) (x)^p$
and $\varphi(B)^{1/q-1/p} \leq \mathcal{M} (\vec{f}) (x)^{p/q-1}$. By (\ref{eq2.8}), we have
$$
\rho (B)\prod_{i=1}^{m}\fint_{B} | f_i (y_i) | dy_i\le C_0\varphi (B)^{1/q-1/p}\prod_{i=1}^{m}\fint_{B} | f_i (y_i)| dy_i\le C_0\mathcal{M}(\vec{f})(x)^{p/q}.
$$
If $r >u$, then $\varphi(B)= \varphi(x,r) < \mathcal{M} (\vec{f})(x)^p$ and $\varphi(B)^{1/q} < \mathcal{M} (\vec{f}) (x)^{p/q}$. By H\"{o}lder's inequality and (\ref{eq1.13}), (\ref{eq2.8}),  we have
\begin{equation*}
\begin{split}
\begin{aligned}
\rho (B)\prod_{i=1}^{m}\fint_{B}\left | f_i\right | dy_i
&\le \rho (B)\prod_{i=1}^{m}\Big ( \fint_{B}\left | f_i \right | ^{p_i}dy_i \Big ) ^{1/{p_i}}
\le \rho (B)\varphi (B)^{1/p}\prod_{i=1}^{m}\left \| f_i \right \|_{L^{({p_i},\varphi_i )}}
\\&\le C_0\varphi (B)^{1/q}\le C_0\mathcal{M}(\vec{f})(x)^{p/q}.
\end{aligned}
\end{split}
\end{equation*}
Then we have (\ref{eq2.9}) and complete the proof.
\end{proof}

\section{Main lemmas\label{sec3}}
In this section we give several lemmas to prove main  results.
\begin{lemma}\label{lm3.1}\,\
Under the assumption in Theorem $\ref{the1.6}~(i)$. For all $f_i\in L^{(p_i,\varphi_i)}(\mathbb{R}^n),\,i=1,\ldots m$ and all balls $ B=B(z,r),\,x\in B$, we have
\begin{equation}\label{eq3.1}
\int_{(\mathbb{R} ^n)^m}\big| K(x,\vec{y} )\vec{f^{\infty}}\big| d\vec{y} \lesssim \int_{2r}^\infty\frac{\varphi (z,t)^{1/p}}{t}dt\prod_{i=1}^{m}  \| f_i  \|_{L^{(p_i,\varphi _i)}},
\end{equation}
\begin{equation}\label{eq3.2}
\int_{(\mathbb{R} ^n)^m}\big| K(x,\vec{y} )\widetilde{\sum }\prod_{i=1}^{m} f _i^{\alpha_i}\big| d\vec{y}
 \lesssim\int_{2r}^{\infty }\frac{\varphi (z,t)^{1/p} }{t}dt\prod_{i=1}^{m}  \| f_i  \|_{L^{(p_i,\varphi _i)}}.
\end{equation}
Moreover, for all $x\in \mathbb{R}^n$, then $T(\vec{f} )(x)$ in $(\ref{eq1.4})$ is well defined. Moreover $T(\vec{f})(x)$ in $(\ref{eq1.4})$  is independent of
the choice of the ball $B$ containing $x$ and we have
\begin{equation}\label{eq3.3}
\|T(\vec{f} ) \| _{L^{(p,\varphi)}}
\lesssim \prod_{i=1}^{m} \| f_i\| _{L^{(p_i,\varphi _i)}}.
\end{equation}
\begin{proof}
For $(\ref{eq3.1})$. If $x\in B(z,r)$ and $y_i \notin 2B$, then $|z-y_i| /2\le  | x-y_i|\le (3/2)| z-y_i| $. From (\ref{eq1.1}) it follows that $| K(x,y_1,\dots ,y_m) | \lesssim (\sum_{i=1}^{m}| x-y_i |)^{-mn}\sim (\sum_{i=1}^{m}| z-y_i |)^{-mn}$. Then
\begin{align*}
\int_{(\mathbb{R} ^n)^m}| K(x,y_1,\dots,y_m )\vec{f^{\infty}} | d\vec{y}
&\lesssim \int_{(\mathbb{R} ^n\setminus 2B)^m}\frac{\prod_{i=1}^{m}| f_i| }{(\sum_{i=1}^{m}| z-y_i |)^{mn}} d\vec{y}
\\&= \sum_{k=0}^{\infty }  \int_{(2^{k+2}B)^m\setminus (2^{k+1}B)^m}\frac{\prod_{i=1}^{m}|f_i| }{(\sum_{i=1}^{m}| z-y_i|)^{mn}}d\vec{y} .
\end{align*}

Since $ (y_1, \ldots , y_m) \in (2^{k+2}B)^m \backslash(2^{k+1}B)^m$, there exists $i_0,\, 1 \leq i_0 \leq m$ such that $y_{i_0}\notin 2^{k+1}B$, which
yields $|z-y_{i_0}|> 2^{k+1}r$, so that $\sum_{i=1}^{m} | z-y_i|> 2^{k+1}r$. By H\"{o}lder's inequality and (\ref{eq1.9}), (\ref{eq1.13}), (\ref{eq1.15}),  we obtain
\begin{align*}
\sum_{k=0}^{\infty } \int_{(2^{k+2}B)^m\backslash(2^{k+1}B)^m } \frac{\prod_{i=1}^{m} | f_i | }{(\sum_{i=1}^{m}  | z-y_i  |)^{mn}  } d\vec{y}
&\leq\sum_{k=0}^{\infty } \int_{(2^{k+2}B)^m } \frac{\prod_{i=1}^{m} | f_i | }{(2^{k+1}r)^{mn}} d\vec{y}
\\&\lesssim \sum_{k=0}^{\infty }\varphi(z,2^{k+2}r)^ {1/p}\prod_{i=1}^{m}  \| f_i  \|_{L^{(p_i,\varphi _i)}}
\\&\lesssim \sum_{k=0}^{\infty }\int_{2^{k+1}r}^{2^{k+2}r}\frac{\varphi (z,t)^{1/p}}{t}dt
\prod_{i=1}^{m}  \| f_i  \|_{L^{(p_i,\varphi _i)}}
\\&\lesssim \int_{2r}^{\infty} \frac{\varphi(z,t)^ {1/p}}{t}dt\prod_{i=1}^{m}  \| f_i  \|_{L^{(p_i,\varphi _i)}}.
\end{align*}
Therefore, we have $(\ref{eq3.1})$. Similarly, we have $(\ref{eq3.2})$.

For $(\ref{eq3.3})$, taking $ B^*=2B$. By $(\ref{eq1.6})$, we get
\begin{align*}
T(\vec{f})(x)=T(\vec{f^0}  )(x)+T(\vec{f^\infty}  )(x)+\widetilde{\sum} T\left(f_{1}^{\alpha_{1}}, \ldots, f_{m}^{\alpha_{m}}\right)(x).
\end{align*}

For $T(\vec{f^0}  )(x)$, by the boundedness of $T$ on $L^p(\mathbb{R}^n)$ and (\ref{eq1.9}), (\ref{eq1.13}), we have
\begin{align*}
 \| T(\vec{f^0}  ) \|_{L^{(p,\varphi )}} &=\sup_B \Big\{ \frac{1}{\varphi (B) } \fint_{B}  | T(\vec{f^0}  )(x) | ^p dx    \Big\}^{1/p}
\\&\lesssim\sup_B \Big\{ \frac{1}{\varphi (B)| B|  } \prod_{i=1}^{m} \|  f_i\chi _{B^*} \|_{L^{p_i}}^p\Big\}^{1/p}
\\&\le \prod_{i=1}^{m}\sup_B \Big\{ \frac{1}{\varphi_i (B)   }  \fint\limits_{B} | f_i\chi _{B^*}  |^{p_i}   \Big \}^{1/{p_i} }
\\&\lesssim\prod_{i=1}^{m} \|f_i\| _{L^{{(p_i,\varphi _i)}}}.
\end{align*}

For $T(\vec{f^\infty})(x)$  and $\widetilde{\sum} T\left(f_{1}^{\alpha_{1}}, \ldots, f_{m}^{\alpha_{m}}\right)(x)$, by $(\ref{eq3.1})$, $~(\ref{eq3.2})$, we obtain
$$ | T(\vec{f^\infty})(x)|,\,\,\widetilde{\sum} \big|T(f_{1}^{\alpha_{1}}, \ldots, f_{m}^{\alpha_{m}})(x)\big|
\lesssim\int_{2r}^{\infty} \frac{\varphi (x,t)^{1/p}}{t} dt\prod_{i=1}^{m} \| f_i  \| _{L^{(p_i,\varphi _i)}},$$
then, by Lemma $\ref{lem2.3}$ and (\ref{eq1.9}), we get
\begin{align*}
\|T(\vec{f^\infty }  )   \| _{L^{(p,\varphi)}}
&\le\sup_B \Big\{ \frac{1}{\varphi (B)} \fint_{B}  \Big| \int_{2r}^{\infty} \frac{\varphi (x,t)^{1/p}}{t} dt \Big|^p dx \Big\}^{1/ p }\prod_{i=1}^{m} \| f_i  \| _{L^{(p_i,\varphi _i)}}
\\&\lesssim \prod_{i=1}^{m} \| f_i \| _{L^{(p_i,\varphi _i)}}.
\end{align*}
Similarly, we obtain
$$\|T(f_{1}^{\alpha_{1}}, \ldots, f_{m}^{\alpha_{m}})\| _{L^{(p,\varphi)}}\lesssim \prod_{i=1}^{m} \| f_i \| _{L^{(p_i,\varphi _i)}}.$$
It follows that
$$\| \sum T( f_{1}^{\alpha_{1}}, \ldots f_{m}^{\alpha_{m}}) \|_{L^{(p,\varphi )}}\lesssim \prod_{i=1}^{m}\| f_i \| _{L^{(p_i,\varphi _i)}}.$$
Therefore, we have $(\ref{eq3.3})$.
\end{proof}
\end{lemma}
\begin{lemma}\label{lm3.2}\,\
Under the assumption in Theorem $\ref{the1.6}~(i)$. For all $b_j\in \mathcal{L}^{(1,\psi)}(\mathbb{R}^n)$, $f_i\in L^{(p_i,\varphi_i)}(\mathbb{R}^n),\,i,j=1,\ldots m,$ and all balls $ B=B(z,r),\,x\in B$, we have
\begin{equation}\label{eq3.4}
\begin{split}
\int_{(\mathbb{R} ^n)^m}\big|(b_j-b_{B^*}^j) K(x,\vec{y})\vec{f^{\infty}}\big| d\vec{y}
&\lesssim\int_{r}^{\infty }\frac{\psi(z,t) }{t}  \Big( \int_{t}^{\infty }\frac{\varphi (z,u)^{1/p} }{u}du \Big) dt
\\&\qquad\times\| b_j  \|_{\mathcal{L}^{(1,\psi )} } \prod_{i=1}^{m} \| f_i  \|_{L^{(p_i,\varphi _i)}},
\end{split}
\end{equation}
\begin{equation}\label{eq3.5}
\begin{split}
\int_{(\mathbb{R} ^n)^m}\big|(b_j-b_{B^*}^j) K(x,\vec{y})\widetilde{\sum} \prod_{i=1}^{m}f_{i}^{\alpha_{i}} \big| d\vec{y}
&\lesssim\int_{r}^{\infty }\frac{\psi(z,t) }{t}   \Big( \int_{t}^{\infty }\frac{\varphi (z,u)^{1/p} }{u}du  \Big) dt
\\&\qquad\times \| b_j  \|_{\mathcal{L}^{(1,\psi )} } \prod_{i=1}^{m}  \| f_i \|_{L^{(p_i,\varphi _i)}},
\end{split}
\end{equation}
where  $b_{B^*}^j= \fint_{B^*}b_j(y_j)dy_j$.
\begin{proof}
For $(\ref{eq3.4})$. If $ x\in B(z,r),\, y_i\notin 2B $, then $|z-y_i|/2 \leq |x-y_i|\leq(3/2)|z-y_i|$ and $|x-y_i|\sim| z-y_i|,\, i=1, \ldots, m $. Since $ (y_1, \ldots , y_m) \in (2^{k+2}B)^m \backslash(2^{k+1}B)^m$, there exists $i_0,\, 1 \leq i_0 \leq m$ such that $y_{i_0}\notin 2^{k+1}B$, which yields $|z-y_{i_0}|> 2^{k+1}r$, so that $\sum_{i=1}^{m} | z-y_i|> 2^{k+1}r$. By H\"{o}lder's inequality and Lemma $\ref{lem2.2}$, (\ref{eq1.9}), (\ref{eq1.13}), we obtain
\begin{align*}
&\int_{(\mathbb{R} ^n)^m}\big|(b_j(y_j)-b_{B^*}^j) K(x,y_1,\dots ,y_m)\vec{f^{\infty} } \big| d\vec{y}
\\&\qquad\lesssim\sum_{k=0}^{\infty }\int_{(2^{k+2}B)^m}  \frac{\prod_{i=1}^{m}| b_j(y_j)-b_{B^*}^j|^{\delta_{ij} } | f_i |  }{(2^{k+2}r)^{mn} }d\vec{y}
\\&\qquad\leq\sum_{k=0}^{\infty }\prod_{i=1}^{m}\fint_{2^{k+2}B} | b_j(y_j)-b_{B^*}^j|^{\delta_{ij} } | f_i |  dy_i
\\&\qquad\leq\sum_{k=0}^{\infty } \Big( \fint_{2^{k+2}B}| b_j(y_j)-b_{B^*}^j|^{p_j'} dy_j  \Big) ^{1/p_j'}\prod_{i=1}^{m}   \Big( \fint_{2^{k+2}B}| f_i |^{p_i}dy_i   \Big) ^{1/{p_i}}
\\&\qquad\lesssim\sum_{k=0}^{\infty }\int_{r}^{2^{k+1}r} \frac{\psi (z,t)}{t}dt \varphi (z,2^{k+2}r)^{1/p} \| b_j \|_{\mathcal{L}^{(1,\psi )} } \prod_{i=1}^{m}  \| f_i \|_{L^{(p_i,\varphi _i)}}
\\&\qquad\lesssim\sum_{k=0}^{\infty }\int_{2^{k+1}r}^{2^{k+2}r}  \Big(\int_{r}^{u}\frac{\psi (z,t)}{t}dt   \Big ) \frac{\varphi  (z,u)^{1/p}}{u}du \| b_j \|_{\mathcal{L}^{(1,\psi )} } \prod_{i=1}^{m} \| f_i  \|_{L^{(p_i,\varphi _i)}}
\\&\qquad=\int_{r}^{\infty }\frac{\psi (z,t)}{t} \Big( \int_{t}^{\infty} \frac{\varphi  (z,u)^{1/p}}{u}du  \Big) dt \| b_j  \|_{\mathcal{L}^{(1,\psi )} } \prod_{i=1}^{m}  \| f_i \|_{L^{(p_i,\varphi _i)}},
\end{align*}
where $\delta _{ij}=\begin{cases}
  &1, \text{ if }, i=j, \\
  & 0,\text{ if }, i\neq j.
\end{cases}$ Therefore, we have $(\ref{eq3.4})$.

For $(\ref{eq3.5})$, we consider $\alpha_1,\ldots,\alpha_m$ such that $\alpha_{j_1} = \ldots = \alpha_{j_l} = \infty$, for some $\{j_1, \ldots , j_l\} \subset\{1, \ldots , m\}$, where
$1 \leq l < m$. Without loss of generality, we consider only the case $\alpha_1 = \ldots = \alpha_s = \infty,\, 1 \leq s < m$, since the other ones follow in analogous way. Since $x ,\,z\in B,\, (y_1, \ldots , y_s) \in (2^{k+2}B)^s \backslash(2^{k+1}B)^s$, there exists $i_0,\, 1 \leq i_0 \leq s$ such that $y_{i_0}\notin 2^{k+1}B$, which yields $| z-y_{i_0}|> 2^{k+1}r$, so that $\sum_{i=1}^{s}|x-y_i|\sim\sum_{i=1}^{s} | z-y_i|> 2^{k+1}r$. By H\"{o}lder's inequality and Lemma $\ref{lem2.2}$, (\ref{eq1.9}), (\ref{eq1.13}), we obtain
\begin{align*}
&\int_{(\mathbb{R} ^n)^m}|(b_j(y_j)-b_{B^*}^j) K(x,y_1,\dots ,y_m)\prod_{i=1}^{m}f_{i}^{\alpha_{i}} | d\vec{y}
\\&\qquad\leq\sum_{k=0}^{\infty }\int_{(2^{k+2}B)^s}  \int_{( 2B)^{m-s}}\frac{\prod_{i=1}^{m}| b_j(y_j)-b_{B^*}^j|^{\delta_{ij} } | f_i |  }{(2^{k+2}r)^{mn} }d\vec{y}
\\&\qquad\leq\sum_{k=0}^{\infty }\int_{(2^{k+2}B)^s}  \int_{(2^{k+2}B)^{m-s}}\frac{\prod_{i=1}^{m}| b_j(y_j)-b_{B^*}^j|^{\delta_{ij} } | f_i |  }{(2^{k+2}r)^{mn} }d\vec{y}
\\&\qquad=\int_{r}^{\infty }\frac{\psi (z,t)}{t} \Big( \int_{t}^{\infty} \frac{\varphi  (z,u)^{1/p}}{u}du  \Big) dt \| b_j  \|_{\mathcal{L}^{(1,\psi )} } \prod_{i=1}^{m}  \| f_i \|_{L^{(p_i,\varphi _i)}}.
\end{align*}
Therefore, we have $(\ref{eq3.5})$
\end{proof}
\end{lemma}

To show that definition (\ref{eq1.8}) is well defined, we give the following remark.
\begin{remark}\label{rm3.3}\,\
Under the assumption in Theorem $\ref{the1.6}~(i)$, let $b_i \in \mathcal{L}^{(1,\psi)}(\mathbb{R}^n)$ and
$f_i \in L^{(p_i,\varphi_i)}(\mathbb{R}^n),\,i=1,\ldots m$. Then $f_i$ is in $L_{\mathrm{loc}}^ {p_i}(\mathbb{R}^n)$ and $b_if_i$ is in $L_{\mathrm{loc}}^{ q_i}(\mathbb{R}^n)$ for all $q_i < p_i,\,i=1,\ldots,m,$ by Lemma $\ref{lem2.1}$. Hence, $T ( \vec{f^\infty} )$ and $T ({b_i}\vec{f^\infty})$ are well defined for any ball $B = B(z,r)$. By $(\ref{eq1.14})$, $(\ref{eq1.15})$ and Lemma $\ref{lem2.3}$, we have,
$$
\int_{r}^{\infty } \frac{\varphi(z,t)^{1/p} }{t} dt\lesssim \varphi(z,r)^{1/p},
$$
and
\begin{align*}
\int_{r}^{\infty } \frac{\psi(z,t) }{t} \Big ( \int_{t}^{\infty } \frac{\varphi(z,u)^{1/p} }{u} du\Big )dt
&\lesssim \int_{r}^{\infty }\frac{\psi(z,t)\varphi(z,t)^{1/p} }{t}dt
\\&\lesssim\int_{r}^{\infty }\frac{\varphi(z,t)^{1/q} }{t}dt
\lesssim\varphi(z,r)^{1/q}.
\end{align*}
Then, by Lemma $\ref{lm3.1}\sim\ref{lm3.2}$, and $(\ref{eq1.6})$, the integrals
$$ T\big(|\vec{f^\infty}|\big)(x),~~\int_{(\mathbb{R} ^n)^{m}}\big| b_j(y_j)K(x,\vec{y} )\vec{f^\infty}\big|  d\vec{y},$$
and
$$ T\Big(\big|\widetilde{\sum }\prod_{i=1}^{m} f _i^{\alpha_i}\big|\Big)(x),~~\int_{(\mathbb{R} ^n)^{m}}\big| b_j(y_j)K(x,\vec{y} )\widetilde{\sum }\prod_{i=1}^{m} f _i^{\alpha_i}\big|  d\vec{y},$$
converge for all $j=1,\ldots,m$. That is, the integrals
$$T_{\vec{b}}^j(\vec{f})(x)=[b_j,T](\vec{f^0})(x)+[b_j,T](\vec{f^\infty})(x)+\widetilde{\sum } [b_j,T](f_1^{\alpha_1}\ldots f_m^{\alpha_m})(x),\,\forall x\in B,$$
is well defined, where $j=1,\ldots,m$. Moreover, if $x \in B_1 \cap B_2$, taking $B_3$ such that $B_1 \cup B_2 \subset B_3$, for all $j=1,\ldots,m$, then
\begin{align*}
&\Big\{[b_j,T](\vec{f^0})_{k,l,l}(x)+[b_j,T](\vec{f^\infty})_{k,l,l}(x)+[b_j,T]\Big(\widetilde{\sum} \prod_{i=1}^{m}f_{i}^{\alpha_{i}}(y_{i})\Big)_{k,l,l}(x) \Big\}
\\&\qquad-\Big\{[b_j,T](\vec{f^0})_{k,3,l}(x)+[b_j,T](\vec{f^\infty})_{k,3,l}(x)+[b_j,T]\Big(\widetilde{\sum} \prod_{i=1}^{m}f_{i}^{\alpha_{i}}(y_{i})\Big)_{k,3,l}(x)\Big\}
\\&\quad=-[b_j,T]\Big(f_k(y_{k})\chi_{(2 B_{3} \backslash 2 B_{l})}(y_{k}){\sum_{\alpha_{1}, \ldots,\alpha_{k-1},\alpha_{k+1},\ldots ,\alpha_{m} \in\{0, \infty\}}} \prod_{i=1,i\neq k}^{m}f_{i}^{\alpha_{i}}(y_{i})\Big)(x)
\\&\qquad+[b_j,T]\Big(f_k(y_{k})\chi_{( 2 B_{l}\backslash 2 B_{3})}(y_{k}){\sum_{\alpha_{1}, \ldots,\alpha_{k-1},\alpha_{k+1},\ldots ,\alpha_{m} \in\{0, \infty\}}} \prod_{i=1,i\neq k}^{m}f_{i}^{\alpha_{i}}(y_{i})\Big)(x)
\\&\qquad-[b_j,T]\Big(f_k(y_{k})\chi_{(2 B_{3} \backslash 2 B_{l})^\complement}(y_{k}){\sum_{\alpha_{1}, \ldots,\alpha_{k-1},\alpha_{k+1},\ldots ,\alpha_{m} \in\{0, \infty\}}} \prod_{i=1,i\neq k}^{m}f_{i}^{\alpha_{i}}(y_{i})\Big)(x)
\\&\qquad+[b_j,T]\Big(f_k(y_{k})\chi_{( 2 B_{l}\backslash 2 B_{3})^\complement}(y_{k}){\sum_{\alpha_{1}, \ldots,\alpha_{k-1},\alpha_{k+1},\ldots ,\alpha_{m} \in\{0, \infty\}}} \prod_{i=1,i\neq k}^{m}f_{i}^{\alpha_{i}}(y_{i})\Big)(x)
=0,
\end{align*}
where $l=1,2.\, (\cdot)_{k,l,s}$ mean to only  $f_k$ is decompose by $B_l$, the others $f_i(i\neq k)$ are decompose by $B_s$  in $\vec{f}=(f_1,\ldots,f_m)$. That is,
\begin{align*}
&\Big\{[b_j,T](\vec{f^0})_{1}(x)+[b_j,T](\vec{f^\infty})_{1}(x)+[b_j,T]\Big(\widetilde{\sum} \prod_{i=1}^{m}f_{i}^{\alpha_{i}}(y_{i})\Big)_{1}(x) \Big\}
\\&\qquad=\Big\{[b_j,T](\vec{f^0})_{2}(x)+[b_j,T](\vec{f^\infty})_{2}(x)+[b_j,T]\Big(\widetilde{\sum} \prod_{i=1}^{m}f_{i}^{\alpha_{i}}(y_{i})\Big)_{2}(x)\Big\}.
\end{align*}
where $(\cdot)_{l}$ mean to $f_i,\,i=1,\ldots,m$ are decompose by $B_l,\,l=1,2$, in $\vec{f}$.

This shows that $[\vec{b},T](\vec{f}) (x)$ in $(\ref{eq1.8})$ is independent of the choice of the ball $B$ containing $x$, since $i,\,j=1,\ldots,m$.
\end{remark}
\begin{lemma}\label{lm3.4}\,\
Under the assumption of Theorem $\ref{the1.6}~(i)$. Then, for all $b_j \in \mathcal{L}^{(1,\psi)}(\mathbb{R}^n)$, $f_i \in L^{(p_i,\varphi_i)}(\mathbb{R}^n),\,i,\,j=1,\ldots m$ and all balls $B = B(z,r)$, we have
\begin{align*}
 \Big| \fint _{B} \Big\{  \int_{((2B)^\complement)^m } K(x,\vec{y})  (  b_j(x)-b_j(y_j)  ) \vec{f} (y)d\vec{y}  \Big\}  dx  \Big|
\lesssim \varphi(z,r)^{1 / q}\|b_j\|_{\mathcal{L}^{(1, \psi)}} \prod_{i=1}^{m} \| f_i\|_{L^{(p_i, \varphi_i )}}.
\end{align*}
\begin{proof}
For $x \in B(z,r)$, let $b_{B}^j=\fint_B b_j(y_j)dy_j$, it follows
\begin{align*}
 \Big| \int_{(\mathbb{R} ^n)^m } K(x,\vec{y})  [b_j(x)-b_j(y_j)]\vec{f} (y)d\vec{y}  \Big|
&\le  | b_j(x)-b_{B}^j   |  \int_{((2B)^\complement)^m } | K(x,\vec{y}) |   \mid \vec{ f} (y)\mid  d\vec{y}
\\&\qquad+\int_{((2B)^\complement)^m } | b_j(y_j)-b_{B}^j |  | K(x,\vec{y})  |   \mid \vec{ f} (y)\mid d\vec{y}
\\&=:G_1(x)+G_2(x).
\end{align*}

For $G_1(x)$, by Lemma $\ref{lm3.1}$, we obtain
$$G_1(x)\leq\left | b_j(x)-b_{B}^j\right |\int_{2r}^{\infty } \frac{\varphi (z,t)^{1/p}}{t} dt\prod_{i=1}^{m}\| f_i \|_{L^{(p_i, \varphi_i )}}.$$
Then, by Lemma $\ref{lem2.3}$,$~(\ref{eq1.9})$, $(\ref{eq1.13})$, $(\ref{eq1.14})$, we obtain
\begin{align*}
\fint_{B}G_1(x)dx&\lesssim \fint_{B}\left | b_j(x)-b_{B}^j\right |\int_{2r}^{\infty } \frac{\varphi (z,t)^{1/p}}{t} dtdx\prod_{i=1}^{m}\| f_i \|_{L^{(p_i, \varphi_i )}}
\\&\lesssim\fint_{B}\left | b_j(x)-b_{B}^j\right |\varphi (z,r)^{1/p}dx\prod_{i=1}^{m}\| f_i \|_{L^{(p_i, \varphi_i )}}
\\&\lesssim \psi  (z,r)\varphi (z,r)^{1/p} \left \| b_j \right \|_{\mathcal{L} ^{(1,\psi )}}\prod_{i=1}^{m}\| f_i \|_{L^{(p_i, \varphi_i )}}
\\&\lesssim \varphi (z,r)^{1/q}\left \| b \right \|_{\mathcal{L} ^{(1,\psi )}}\prod_{i=1}^{m}\| f_i \|_{L^{(p_i, \varphi_i )}}.
\end{align*}

For $G_2(x)$, by Lemma $\ref{lem2.3}$, Lemma $\ref{lm3.2}$,  $(\ref{eq1.9})$, $(\ref{eq1.13})$, $(\ref{eq1.14})$, we have
$$\fint_{B}G_2(x)dx\leq\varphi (z,r)^{1/q} \left \| b_j \right \|_{\mathcal{L} ^{(1,\psi )}}\prod_{i=1}^{m} \| f_i \|_{L^{(p_i , \varphi_i  )}}.$$
Combining the methods of estimating $G_1(x)$ and $G_2(x)$, we obtain the desired estimate.
\end{proof}
\end{lemma}
\begin{lemma}\label{lm3.5}\,\
Let $m \in N$ and $\vec{b} = (b_1,...,b_m)$ be a collection of locally integrable
functions. For any $B\subset \mathbb{R}^n $, the following statements are equivalent:\\
$(a)$ There exists a constant $C_1$ such that
$$[\vec{b}]_{*}:=\sup _{B} \frac{1}{\psi(B) |B|^{m+1}} \int_{B} \int_{B^{m}}\Big|\sum_{j=1}^{m}(b_{j}(x)-b_{j}(y_{j}))\Big| d \vec{y} d x \leq C_1,$$
$(b)$ There exists a constant $C_2$ such that
$$[\vec{b}]_{* *}:=\sup _{B} \frac{1}{\psi(B) |B|^{m}} \int_{B^{m}}\Big|\sum_{j=1}^{m}(b_{j}(x_{j})-b_B^j)\Big| d \vec{x} \leq C_2,$$
$(c)$ There exists a constant $C_3$ such that
$$[\vec{b}]_{* * *}:=\sup _{B} \frac{1}{\psi(B) |B|^{2 m}} \int_{B^{m}} \int_{B^{m}}\Big|\sum_{j=1}^{m}(b_{j}(x_{j})-b_{j}(y_{j}))\Big| d \vec{y} d \vec{x} \leq C_3,$$
$(d)$ $b_1,\ldots,b_m \in \mathcal{L}^{(1,\psi)}(\mathbb{R}^n).$
\begin{proof}
$(a)\Longrightarrow(b).$
\begin{align*}
\frac{1}{\psi(B) |B|^{m}}\int_{B^{m}}\Big|\sum_{j=1}^{m}(b_{j}(x_{j})-b_B^j)\Big| d \vec{x}&=\frac{1}{\psi(B) |B|^{m}}\int_{B^{m}}\Big|\sum_{j=1}^{m}(b_{j}(y_{j})-b_B^j)\Big| d \vec{y}
\\&=\frac{1}{\psi(B) |B|^{m}}\int_{B^{m}}\Big|\sum_{j=1}^{m}(b_{j}(y_{j})-\frac{1}{| B| } \int_{B}^{} b_{j}(x)dx)\Big| d \vec{y}
\\&=\frac{1}{\psi(B) |B|^{m+1}}\int_{B}\int_{B^{m}} \Big|\sum_{j=1}^{m}(b_{j}(y_{j})- b_{j}(x))\Big| d \vec{y}dx.
\end{align*}
$(b)\Longrightarrow(c).$
\begin{align*}
&\frac{1}{\psi(B) |B|^{2m}}\int_{B^{m}}\int_{B^{m}}\Big|\sum_{j=1}^{m}(b_{j}(x_{j})-b_{j}(y_j))\Big| d \vec{y}d \vec{x}
\\&\qquad=\frac{1}{\psi(B) |B|^{2m}}\int_{B^{m}}\int_{B^{m}}\Big|\sum_{j=1}^{m}(b_{j}(x_{j})-b_B^j+b_B^j-b_{j}(y_j))\Big| d \vec{y}d \vec{x}
\\&\qquad\leq\frac{1}{\psi(B) |B|^{m}}\int_{B^{m}}\Big|\sum_{j=1}^{m}(b_{j}(x_{j})-b_B^j)\Big| d \vec{x}
\\&\qquad\qquad+\frac{1}{\psi(B) |B|^{m}}\int_{B^{m}}\Big|\sum_{j=1}^{m}(b_{j}(y_j)-b_B^j)\Big| d \vec{y}\lesssim C_2.
\end{align*}
$(c)\Longrightarrow(d).$
We denote  $$\Omega_{m}=\big\{\vec{\sigma}_{m}=(\sigma_{1}, \ldots, \sigma_{m}): \sigma_{j} \in\{-1,1\}, i=1, \ldots, m\big\}.$$
For any $a_j \in  \mathbb{R}^n$, \cite{DWZT} establish   as following inequality,
\begin{equation}\label{eq3.7}
\sum_{j=1}^{m} | a_j| \le \sum_{\vec{\sigma}_m\in \Omega_m } \Big| \sum_{j=1}^{m}\sigma_j a_j  \Big|.
\end{equation}
Applying the inequality $(\ref{eq3.7})$, we obtain that for any ball $B$,
\begin{align*}
\frac{1}{|B|^{2 m}} \int_{B^{2 m}} \sum_{j=1}^{m}|b_{j}(x_{j})-b_{j}(y_{j})| d \vec{x} d \vec{y}
&\leq \sum_{\vec{\sigma}_{k+1} \in \Omega_{k+1}} \frac{1}{|B|^{2 m}} \int_{B^{2 m}}\Big|\sum_{j=1}^{m} \sigma_{j}(b_{j}(x_{j})-b_{j}(y_{j}))\Big| d \vec{x} d \vec{y}
\\&\leq \sum_{\vec{\sigma}_{k+1} \in \Omega_{k+1}}[\vec{b}]_{* * *}\leq C[\vec{b}]_{* * *},
\end{align*}
which yields that for $j = 1,\ldots,m$,
\begin{align*}
&\frac{1}{\psi(B) }\fint_{B}| b_{j}(x_{j})-b_B^{j}| d {x_{j}}
\le \frac{1}{\psi(B) |B|^2}\int_{B^2}| b_{j}(x_{j})-b{(y_j)}| d {x_{j}}d {y_{j}}\le C[\vec{b}]_{* * *}.
\end{align*}
Then, $b_1,\ldots,b_m \in \mathcal{L}^{(1,\psi)}(\mathbb{R}^n)$.

$(d)\Longrightarrow(a)$.  For any $B$, we have
\begin{align*}
\frac{1}{\psi(B) |B|^{m+1}} \int_{B} \int_{B^{m}}|\sum_{j=1}^{m}(b_{j}(x)-b_{j}(y_{j}))| d \vec{y} d x
&\le \sum_{j=1}^{m}\frac{1}{\psi(B) } \fint_{B} \fint_{B}| b_{j}(x)-b_{j}(y_{j})|dy_jdx
\\&\le \sum_{j=1}^{m} \frac{1}{\psi(B)} \fint_{B} | b_{j}(x)-b_B^{j}|  d x
\\&\qquad+\sum_{j=1}^{m}\frac{1}{\psi(B)}  \fint_{B}| b_{j}(y_{j})-b_B^{j}| dy_{j}
\\&\lesssim \|\vec{b}\|_{(\mathcal{L}^(1,\psi))^m}. \end{align*}
\end{proof}
\end{lemma}

\section{Sharp maximal operator and pointwise estimate\label{sec4}}
In this section, we will establish sharp maximal inequality and pointwise estimate.

For $f\in L_{\mathrm{loc}}^1(\mathbb{R} ^n)$, let
$$
M^{\sharp } f(x)=\sup _{B \ni x} \fint _{B}|f(y)-f_{B}| d y, \, x \in \mathbb{R}^{n},
$$
where the supremum is taken over all balls $B$ containing $x$.
\begin{proposition}\label{pro7.2}\,\
Let $p,\,p_i,\,\eta \in (1,\infty)$ satisfies $\sum_{i=1}^{m} 1/{p_i} =1/p$  and $\varphi,\,\varphi_i,\,\psi :\mathbb{R}^n\times (0,\infty )\rightarrow (0,\infty )$. Let $T$ be an $m$-linear Calder\'{o}n-Zygmund operators with kernel satisfies Definition $\ref{def1.1}$.
Assume that $\varphi,\,\varphi_i\in \mathcal{G}^{dec}$ satisfies $(\ref{eq1.13})$, $(\ref{eq1.15})$ and
$ \psi\in \mathcal{G}^{inc}$ satisfies $(\ref{eq1.10})$, that  $\int_{r}^{\infty } \frac{\psi (x,t)\varphi (x,t)^{1/p}}{t} dt<\infty $ and $\int_{0}^{1} \frac{\omega (t)\log{\frac{1}{t} } }{t} dt<\infty $, for each $x \in \mathbb{R} ^n$ and $r>0$. Then there exists a positive constant $C$ such that, for all $ b_j\in \mathcal{L}^{(1,\psi)}(\mathbb{R}^n)$, $f_i\in L^{(p_i,\varphi_i )}(\mathbb{R}^n)$ and $x\in\mathbb{R}^n,\,i,\,j=1,\ldots,m $,
$$
M^{\#}[\vec{b},T](\vec{f})(x) \leq C\|\vec{b}\|_{(\mathcal{L}^{(1, \psi)})^m} \Big\{[ M_{\psi^{\eta }}(| T(\vec{f})|^{\eta})(x) ] ^{1 / \eta}
+ [ \mathcal{M}_{\psi^{\eta }}(| \vec{f} |^{\eta})(x) ] ^{1 / \eta}\Big\},
$$
where $C$ is a positive constant independent of $f_i$ and $b_j$.
\end{proposition}
\begin{proof}
It suffices to prove the theorem for $T_{\vec{b}}^j(\vec{f})(x)$. For any ball $B=B(x,r)$ be a ball centered at $x$. For $ z\in B$, taking $B^*=2B$, by $(\ref{eq1.6})$ and $(\ref{eq1.8})$,  we have
 \begin{align*}
T_{\vec{b}}^j(\vec{f})(z)
&=b_j(z)T(\vec{f})(z)-T\big(f_1,\dots ,f_jb_j,\dots ,f_m\big)(z)
\\&= ( b_j(z)-b_{B^*}^j  ) T(\vec{f})(z)-T\big(f_1,\dots ,f_j(b_j(y_j)-b_{B^*}^j ),\dots ,f_m\big)(z),
\end{align*}
we denote
\begin{align*}
F_1(z)&=(b_j(z)-b_{B^*}^j)T(\vec{f})(z),~~F_2(x)=T\big((b_j(y_j)-b_{B^*}^j)\vec{f^0} \big)(z),
\\F_3(z)&=T\big((b_j(y_j)-b_{B^*}^j)\vec{f^\infty} \big)(z)-T\big((b_j(y_j)-b_{B^*}^j)\vec{f^\infty} \big)(x),
\\F_4(z)&=\widetilde{\sum}T\Big((b_j(y_j)-b_{B^*}^j)\prod_{i=1}^{m} f_{i}^{\alpha _{i}}\Big)(z)-\widetilde{\sum}T\Big((b_j(y_j)-b_{B^*}^j)\prod_{i=1}^{m} f_{i}^{\alpha _{i}}\Big)(x),
\\C_B&=\widetilde{\sum}T\Big((b_j(y_j)-b_{B^*}^j)\prod_{i=1}^{m} f_{i}^{\alpha _{i}}\Big)(x)+T\Big((b_j(y_j)-b_{B^*}^j)\prod_{i=1}^{m} f_{i}^{\infty}\Big)(x),
\end{align*}
where $\widetilde{\sum} $ contains $ \alpha _1,\dots ,\alpha _m $ are not all equal to $0$ or $\infty$ at the same time.
Then, we have
$$T_{\vec{b}}^j(\vec{f})(z)+C_B=:F_1(z)-F_2(z)-F_3(z)-F_4(z).$$
Observe that if suffices to show that
$$\fint_{B}| F_i(z)| dz\le  C\|\vec{b}\|_{(\mathcal{L}^{(1, \psi)})^m} \Big\{  \big[ M_{\psi^{\eta }}(| T(\vec{f})|^{\eta})(x)  \big] ^{1 / \eta}
+ \big[ \mathcal{M}_{\psi^{\eta }}(| \vec{f} | ^{\eta})(x) \big] ^{1 / \eta}  \Big\},$$
where $i=1,2,3,4.$ Then we have the  desired conclusion.

For $F_1(z)$, by  H\"{o}lder's inequality and Lemma $\ref{lem2.1}$, we obtain
\begin{align*}
\fint _{B}| F_{1}(z)| dy& \leq \frac{1}{\psi(B)}\Big(\fint _{B}| b_j(z)-b_{B*}^j |^{\eta^{\prime}} d z\Big)^{1 / \eta^{\prime}} \Big(\psi(B)^{\eta} \fint _{B}| T(\vec{f} )(z)|^{\eta} d z\Big)^{1 / \eta}
\\&\lesssim\|b_j\|_{\mathcal{L}^{(1, \psi)}}\Big(M_{\psi^{n}}(| T(\vec{ f} )|^{\eta})(x)\Big)^{1 / \eta}.
\end{align*}

For $F_2(z)$, choose $v \in (1, \eta),$ satisfies $1/\nu=1/u+1/\eta$. Since by the  boundedness of $T$ on $L^\nu(\mathbb{R}^n)$ and H\"{o}lder's inequality, we have
\begin{align*}
\fint _{B}|F_{2}(z)| dz
&\le\Big( \fint _{B}|F_{2}(z)|^\nu dz\Big)^{1/\nu }
\lesssim \Big( \frac{1}{|B|} \int _{(\mathbb{R} ^n)^m}\big|(b_j(y_j)-b_{B^*}^j)\vec{f^0}\big|^\nu d\vec{y}   \Big)^{1/\nu}
\\&\leq \frac{1}{\psi(2B)}\Big(\fint _{2B}|b_j(y_j)-b_{B^*}^j |^{\eta^{\prime}} d y_j\Big)^{1 / \eta^{\prime}}\Big(\psi(2B)^{\eta} \prod_{i=1}^{m} \fint _{2B}|f_i |^{\eta} d y_i \Big)^{1 / \eta}
\\&\lesssim\|b_j\|_{\mathcal{L}^{(1, \psi)}}\Big(\mathcal{M}_{\psi^{n}}(|\vec{f} |^{\eta})(x)\Big)^{1 / \eta}.
\end{align*}

For $F_3(z)$. Since $z\in B(x,r),\, (y_1, \ldots , y_m) \in (2^{k+2}B)^m\backslash (2^{k+1}B)^m$, there exists $i_0,\, 1 \leq i_0 \leq m$ such that $y_{i_0}\notin 2^{k+1}B$, which
yields $|z-y_{i_0}|> 2^{k+1}r$, so that $\sum_{i=1}^{m} | z-y_i|> 2^{k+1}r$ and $|z-y_i|\sim| x-y_i|$. For $1<\eta<\infty$, by  H\"{o}lder's inequality and Lemma $\ref{lem2.1}$, Lemma $\ref{lm3.2}$ and (\ref{eq1.2}), we have
\begin{align*}
 | F_3(z)| &=\big| T\big((b_j(y_j)-b_{B^*}^j)\vec{f^\infty}\big)(z)-T\big((b_j(y_j)-b_{B^*}^j)\vec{f^\infty}\big)(x)\big|
\\&\le \int_{(\mathbb{R} ^n)^m} | K(z,\vec{y} )-K(x,\vec{y}) |  |(b_j(y_j)-b_{B^*}^j)\vec{f^\infty}  |d\vec{y}
\\&\lesssim\sum_{k=0}^{\infty } \int_{(2^{k+2}B)^m/(2^{k+1}B)^m} \frac{\omega(\frac{ | x-z | }{\sum_{i=1}^{m} | z-y_i |  })}{(\sum_{i=1}^{m}  | x-y_i  | )^{mn }   }
 |(b_j(y_j)-b_{B^*}^j)\vec{f} |d\vec{y}
\\&\leq  \sum_{k=0}^{\infty }\omega  ( 1/{2^{k+2}}  )  \fint_{(2^{k+2}B)^m}   |(b_j(y_j)-b_{B^*}^j)\vec{f}   |d\vec{y}
\\&\le \sum_{k=0}^{\infty }(k+2)\omega  ( \frac{1}{2^{k+2}}  )\|b_j\|_{\mathcal{L}^{(1, \psi)}}\psi (2^{k+2}B)\prod_{i=1}^{m} \Big\{ \fint_{2^{k+2}B} | f_i  |^\eta dy_i \Big \}  ^{1 / \eta}
\\&\lesssim \int_{0}^{1} (\log\frac{1}{t} )\frac{\omega (t)}{t}dt\|b_j\|_{\mathcal{L}^{(1, \psi)}}\Big(\mathcal{M}_{\psi^{n}}(|\vec{f} |^{\eta})(x)\Big)^{1 / \eta}
\\&\lesssim \|b_j\|_{\mathcal{L}^{(1, \psi)}}\Big(\mathcal{M}_{\psi^{n}}(|\vec{f} |^{\eta})(x)\Big)^{1 / \eta}.
\end{align*}

Therefore,
$$\fint_{B}  | F_3(z)  | dz\lesssim\|b_j\|_{\mathcal{L}^{(1, \psi)}}\Big(\mathcal{M}_{\psi^{n}}(|\vec{f} |^{\eta})(x)\Big)^{1 / \eta}.$$

 For $F_4(z)$, we get
\begin{align*}
 | F_4(z)|& =\Big| \widetilde{\sum} T\big((b_j(y_j)-b_{B^*}^j)\prod_{i=1}^{m} f_{i}^{\alpha _{i}}\big)(z)-\widetilde{\sum} T\big((b_j(y_j)-b_{B^*}^j)\prod_{i=1}^{m} f_{i}^{\alpha _{i}}\big)(x)\Big|
\\&\le \widetilde{\sum}\int_{(\mathbb{R} ^n)^m} \big| K(z,\vec{y})-K(x,\vec{y})  \big|\Big| (b_j(y_j)-b_{B^*}^j)\prod_{i=1}^{m} f_{i}^{\alpha _{i}} \Big| d\vec{y}.
\end{align*}
We consider $F_4(z)$ such that $\alpha_{j_1} = \dots =\alpha_{j_l} =0$ for some $\{j_1,\dots j_l\}\subset \{1,\dots,m\},$ where $1\le l<m$. Without loss of generality, we consider only the case $\alpha _1=\dots =\alpha _s=0,\,1\le s<m$, since the other ones follow in analogous way. Then by the similar argument as $F_3(z)$, using (\ref{eq1.2}), we get
\begin{align*}
&\int_{(\mathbb{R} ^n)^m}| K(z,\vec{y})-K(x,\vec{y})  |  | b_j-b_{B^*}^j|\prod_{i=1}^{m} |  f_i^{\alpha _i}| d\vec{y}
\\&\qquad\lesssim   \int_{(\mathbb{R} ^n)^m} \frac{|(b_j(y_j)-b_{B^*}^j) \prod_{i=1}^{m}f_{i}^{\alpha_{i}} |}{(\sum_{i=s+1}^{m} | x-y_i  |)^{mn} }\omega\Big(\frac{ | x-z | }{\sum_{i=1}^{m} | z-y_i |  }\Big)  d\vec{y}
\\&\qquad=\int_{(B^*)^s}\int_{(\mathbb{R}^n\setminus B^*)^{m-s}}  \frac{| b_j(y_j)-b_{B^*}^j| \prod_{i=1}^{m}| f_{i} |}{(\sum_{i=s+1}^{m}  | x-y_i |)^{mn} }\omega\Big(\frac{ | x-z | }{\sum_{i=1}^{m} | z-y_i |  }\Big)d\vec{y}
\\&\qquad\leq\sum_{k=0}^{\infty } \int_{(2^{k+2}B)^s}\int_{(2^{k+2}B)^{m-s}}
\frac{\prod_{i=1}^{m}| b_j(y_j)-b_{B^*}^j|^{\delta_{ij}} | f_{i} |}{(2^{k+2}r)^{mn} }\omega  (  \frac{1 }{2^{k+2}}  ) d\vec{y}
\\&\qquad=\sum_{k=0}^{\infty } \omega  (  \frac{1 }{2^{k+2}}  )\int_{(2^{k+2}B)^m}
\frac{\prod_{i=1}^{m}| b_j(y_j)-b_{B^*}^j|^{\delta_{ij}} | f_{i} |}{(2^{k+2}r)^{mn} } d\vec{y},
\end{align*}
which, together with Lemma $\ref{lem2.1}$ and $\ref{lm3.2}$
leads to
\begin{align*}
&\int_{(\mathbb{R} ^n)^m}| K(z,\vec{y})-K(x,\vec{y}) |  |b_j(y_j)-b_{B^*}^j| \prod_{i=1}^{m} | f_{i}^{\alpha _{i}} | d\vec{y}
\\&\qquad\lesssim\sum_{k=0}^{\infty } \omega  (  \frac{1 }{2^{k+2}}  )  \Big( \fint_{2^{k+2}B}  | b_j(y_j)-b_{B^*}^j  |^{\eta '}  dy_j  \Big)^{1/{\eta'} } \prod_{i=1}^{m}  \Big( \fint_{2^{k+2}B}|  f_i |^\eta dy_i  \Big)^{1/{\eta } }
\\&\qquad\lesssim\sum_{k=0}^{\infty }  (k+2)\omega  ( \frac{1}{2^{k+2}}   ) \| b_j  \|_{\mathcal{L} ^{(1,\psi )}}\psi (2^{k+2}B)\prod_{i=1}^{m} \Big( \fint_{2^{k+2}B} |  f_i |^\eta dy_i \Big)^{1/{\eta } }
\\&\qquad\lesssim\|b_j\|_{\mathcal{L}^{(1, \psi)}}\Big(\mathcal{M}_{\psi^{n}}(|\vec{f}|^{\eta})(x)\Big)^{1 / \eta}.
\end{align*}

Summing up the estimates of $F_1(z),\, F_2(z),\,F_3(z)$ and $F_4(z)$, it immediately yields,
$$M^{\#}[b_j,T](\vec{f})(x) \leq C\|b_j\|_{\mathcal{L}^{(1, \psi)}} \Big\{  \big[ M_{\psi^{\eta }}(| T(\vec{f})|^{\eta})(x) \big] ^{1 / \eta}
+ \big[ \mathcal{M}_{\psi^{\eta }}(| \vec{f} | ^{\eta})(x) \big] ^{1 /\eta}\Big\}.$$
The proposition is proved.
\end{proof}

\section{Proof of the Theorem $\ref{the1.6}$\label{sec5}}
This section is devoted to the proof of Theorem $\ref{the1.6}$.
For sharp maximal operator, the following lemma is known.
\begin{lemma}[\cite{ARNE1}]\label{lm5.1}\,\
Let $p\in[1,\infty)$ and $\varphi :\mathbb{R}^n\times (0,\infty )\rightarrow (0,\infty )$. Assume that $ \varphi\in \mathcal{G}^{dec}$ and that $\varphi$ such that $(\ref{eq1.15})$. For $f\in L_{\mathrm{loc}}^1(\mathbb{R} ^n)$, if $\lim_{r \to \infty} f_{B(0,r)}=0$, then
\begin{equation}\label{eq5.1}
\left \| f \right \| _{L^{(p,\varphi )}}\le C\left \|M^\sharp f \right \| _{L^{(p,\varphi )}},
\end{equation}
where $C$ is a positive constant independent of $f$.
\end{lemma}
For $0<\eta<\infty$, we have
\begin{equation}\label{eq5.2}
\left \| \left | f \right | ^\eta \right \| _{L^{(p,\varphi )}}= \left \| f  \right \| _{L^{(p\eta,\varphi )}}^\eta.
\end{equation}
\begin{proof}[\bf Proof of Theorem$~\ref{the1.6}$(i)]
By the assumption of Theorem$~\ref{the1.6}(i)$ and Lemma $\ref{lm3.1}$, we have
 $$\|T(\vec{f} )\|_{L^{(p,\varphi)}}\leq C\prod_{i=1}^{m} \|f_i\|_{L^{(p_i,\varphi_i)}}.$$
 Let $1<\eta<p$, from (\ref{eq1.14}), we obtain
$$\psi(x,r)^\eta\varphi(x,r)^{\eta/p}\leq{C}_0^\eta\varphi(x,r)^{\eta/q}.$$
Then, by Lemma \ref{lem2.5}, we know that
 $$\|M_{\psi^\eta}(f)\|_{L^{(q/\eta,\varphi)}(\mathbb{R}^n)}\lesssim \|f\|_{L^{(p/\eta,\varphi)}(\mathbb{R}^n)}.$$
This, together with $(\ref{eq3.3})$, leads to
\begin{align*}
\| M _{\psi^{\eta}}(| T (\vec{f} )| ^{\eta})^{1 / \eta}\|_{L^{(q, \varphi)}}
&=\| M_{\psi^{\eta}}(| T(\vec{f} )|^{\eta})\|_{L^{(q / \eta, \varphi)}}^{1 / \eta}
\lesssim\|| T(\vec{f} )|^{\eta}\|_{L^{(p / \eta, \varphi)}}^{1 / \eta}
\\&=\| T (\vec{f} ) \|_{L^{(p, \varphi)}} \lesssim\prod_{i=1}^{m} \| f_i\|_{L^{(p_i, \varphi_i)}},
\end{align*}
and by (\ref{eq5.2}) and Lemma \ref{lem2.7}, we have
\begin{align*}
\|\mathcal{M}_{\psi^{\eta}}(|\vec{f} |^{\eta})^{1 / \eta}\|_{L^{(q, \varphi)}}
=\|\mathcal{M}_{\psi^{\eta }}(|\vec{f} |^\eta )\|_{L^{(q/ \eta, \varphi)}}^{1 / \eta}
\lesssim\||\vec{f} |^{\eta}\|_{L^{(p/ \eta , \varphi)}}^{1 /\eta}
=\prod_{i=1}^{m} \| f_i\|_{L^{(p_i, \varphi_i)}}.
\end{align*}
Then, using  Proposition$~\ref{pro7.2}$, we have
 $$
\| M^{\#}([b_j,T]( \vec{f} ))\|_{L^{(q, \varphi)}} \lesssim\| b_j\|_{\mathcal{L}^{(1,\psi)}}\prod_{i=1}^{m} \|f_i\|_{L^{(p_i, \varphi_i)}},\,1\leq j\leq m.
$$
Therefore, if we show that, for $B_r=B(0,r)$,
\begin{equation}\label{eq5.4}\,\
\fint_{B_r}[b_j,T](\vec{ f} )dx\rightarrow 0,\,as\,r\rightarrow\infty,\,1\leq j\leq m.
\end{equation}
 Then, use Lemma $\ref{lm5.1}$, we have
\begin{equation}\label{eq5.5}\,\
\|[b_j,T](\vec{f} )\|_{L^{(q, \varphi)}} \lesssim\| M^{\#}([b_j,T]( \vec{f} ))\|_{L^{(q, \varphi)}}\lesssim\| b_j\|_{\mathcal{L}^{(1, \psi)} }\prod_{i=1}^{m}\| f_i\|_{L^{(p_i, \varphi_i)}},
\end{equation}
which is the desired conclusion.

It remains to show (\ref{eq5.4}). Since
\begin{align*}
[b_j,T](\vec{f} )(x)=b_j(x)T(\vec{f}\, )(x)-T(b_j\vec{f}\, )(x).
\end{align*}
To obtain (\ref{eq5.4}) it suffices to prove
$$\fint_{B_r} b_j(x)T(\vec{f}\,)(x)dx\rightarrow 0\,and \, \fint_{B_r}T(b_j\vec{f}\,)(x)dx\rightarrow 0,\,as\,r\rightarrow\infty.$$

Without loss of generality, we only consider  $m = 2$ and $j=1$, which is divided into the following three cases.

{\bf Case 1.} First we show (\ref{eq5.4}) for all $ f_1\in L^{(p_1, \varphi_1)}(\mathbb{R}^n)$ and $f_2\in L^{(p_2, \varphi_2)}(\mathbb{R}^n)$ with compact support. Let $\supp f_1\subset B_s=B(0,s)$ and $\supp f_2\subset B_s=B(0,s)$ with $s\geq1,\,B_{2s}=2B_s$. Then $f_1\in L^{p_1}(\mathbb{R}^n),\,f_2\in L^{p_2}(\mathbb{R}^n)$ and $b\in L_{\mathrm{loc}}^{p_0}(\mathbb{R}^n)$ for all $p_0\in(1,\infty)$.
We decompose
$$(1)~~ bT(f_1,f_2)=bT(f_1,f_2)\chi_{B_{2s}}+bT(f_1,f_2)\chi_{(B_{2s})^\complement},$$
$$(2)~~ T(bf_1,f_2)=T(bf_1,f_2)\chi_{B_{2s}}+T(bf_1,f_2)\chi_{(B_{2s})^\complement}.$$
Taking $1/p+1/{p_0}=1$, since $ T $ is bounded on Lebesgue spaces, we obtain
\begin{align*}
\int_{\mathbb{R}^n}| bT(f_1,f_2)\chi_{B_{2s}}dx|&\leq\Big ( \int_{\mathbb{R}^n}  |b|^{p_0 } \chi_{B_{2s}}dx\Big ) ^{1/{p_0} }\Big ( \int_{\mathbb{R}^n}\big|T(f_1,f_2)\big |^{p } \chi_{B_{2s}}dx\Big ) ^{1/p }
\\&\le \|b\|_{L_{\mathrm{loc}}^{p_0}}\|f_1\|_{L^{p_1}}\|f_2 \| _{L^{p_2}},
\end{align*}
and taking ${1}/{q_2}+{1}/{p_2}={1}/{\gamma},\,1/{q_2}=1/{p_0}+1/{p_1}$, we have
\begin{align*}
\int_{\mathbb{R}^n}| T(bf_1,f_2)\chi_{B_{2s}}|
&\le |B_{2s} |\Big(\fint_{B_{2s}}| T(bf_1,f_2)|^\gamma  \Big)^{{1}/{\gamma } }
\\&\lesssim \Big(\int_{B_{2s}}  | bf_1  |^{q_2 } dy_1 \Big) ^{1/{q_2} } \Big(\int_{B_{2s}}|f_2|^{p_2 } dy_2\Big) ^{1/{p_2} }
\\&\lesssim  \Big(\int_{B_{2s}}|b|^{p_0} dy_1 \Big)^{1/{p_0} } \Big( \int_{B_{2s}}|f_1|^{p_1 }dy_1\Big)^{1/{p_1} }\Big( \int_{B_{2s}}|f_2|^{p_2 }dy_2\Big)^{1/{p_2} }
\\&\le  \|b\|_{L_{\mathrm{loc}}^{p_0}(B_{2s})} \| f_1  \|_{L^{p_1}}\| f_2\|_{L^{p_2}}.
\end{align*}
Observe that $T(f_1,f_2)(x)\chi_{B_{2s}}$ and $T(bf_1,f_2)(x)\chi_{B_{2s}}$ are in $ L^{1}(\mathbb{R}^n)$,
then $bT(f_1,f_2)\chi_{B_{2s}},$ and $ T(bf_1,f_2)\chi_{B_{2s}}\in L^1(\mathbb{R}^n)$, which yields
\begin{align*}
\fint_{B_r}| bT(f_1,f_2) | \chi_{B_{2s}}
\leq \frac{1}{|B_r| }\int_{\mathbb{R} ^n}|bT(f_1,f_2)|  \chi_{B_{2s}}
= \frac{1}{| B_r | } \| bT(f_1,f_2) \chi_{B_{2s}} \| _{L^1(\mathbb{R} ^n)}\rightarrow 0,
\end{align*}
$as\,r\rightarrow\infty$. Similarly,
\begin{align*}
\fint_{B_r}| T(bf_1,f_2)\chi_{B_{2s}}|
\leq \frac{1}{ | B_r  | } \int_{\mathbb{R} ^n} | T(bf_1,f_2) | \chi_{B_{2s}}
= \frac{1}{|B_r|}  \| T(bf_1,f_2)\chi_{B_{2s}}\| _{L^1(\mathbb{R} ^n)}\rightarrow 0,
\end{align*}
$as\,r\rightarrow\infty.$

If $x\in {(B_{2s})}^\complement$ and $y_i\in B(0,s)$, then $|x|/2 \leq |x-y_i|\leq(3/2)|x|$, which yields $|x -y_i|\sim~|x|,\,i=1,2$.
By $(\ref{eq1.1})$, we obtain
\begin{equation}\label{eq5.6}
| T(f_1,f_2)(x)|\lesssim \frac{1 }{ | x |^{2n} }\int_{\mathbb{R} ^n}\int_{\mathbb{R} ^n} | f_1(y_1) f_2(y_2) |dy_1dy_2=\frac{1 }{ | x |^{2n} } \| f_1  \| _{L^1} \| f_2  \| _{L^1},
\end{equation}
\begin{equation}\label{eq5.7}
| T(bf_1,f_2)(x)| \lesssim \frac{1}{|x|^{2n} }\|b f_1 \|_{L^1}\|f_2\| _{L^1},
\end{equation}
$$ \| bf_1  \| _{L^1}= \| bf_1 \chi _{B_s} \| _{L^1}\le \| b \chi _{B_s} \| _{L^{p_0}} \|f_1 \| _{L^{p}}.$$
Since
\begin{align*}
\int_{B_r}bT(f_1,f_2)(x)\chi _{(B_{2s})^\complement }(x)dx&=\int_{B_r}(b(x)-b_{B_{2s}}+b_{B_{2s}})T(f_1,f_2)(x)\chi _{(B_{2s})^\complement }(x)dx
\\&=\int_{B_r}(b(x)-b_{B_{2s}})T (f_1,f_2)(x)\chi _{(B_{2s})^\complement }(x)dx
\\&\qquad+b_{B_{2s}}\int_{B_r}T (f_1,f_2)(x)\chi _{(B_{2s})^\complement }(x)dx,
\end{align*}
\begin{align*}
| b_{B_{2s}}|\fint_{B_r}| T (f_1,f_2)\chi _{(B_{2s})^\complement }|dx
&\le| b_{B_{2s}}|\frac{1}{ | B_r  | } \int_{B_r}\frac{1}{ | x |^{2n} } \| f_1  \| _{L^1} \| f_2  \| _{L^1}\chi _{(B_{2s})^\complement }dx
\\&\lesssim | b_{B_{2s}}|\frac{1 }{ | B_r  |^2 } \| f_1 \| _{L^1} \| f_2  \| _{L^1} \rightarrow 0,\,as\,r\rightarrow\infty,
\end{align*}
and
$$
\fint_{B_r}| T(bf_1,f_2)\chi _{(B_{2s})^\complement }|dx\le \fint_{B_r}\frac{1}{|x|^{2n}} \chi _{(B_{2s})^\complement }dx \|bf_1\|_{L^1}\| f_2 \|_{L^1}\to 0,\,as\,r\rightarrow\infty.
$$

Next, we show
\begin{equation}\label{eq5.9}
\fint_{B_{r}}\left(b(x)-b_{B_{2 s}}\right)T(f_1,f_2)(x)\chi_{(B_{2 s})^\complement }(x) \rightarrow 0,\,as\,r \rightarrow \infty.
\end{equation}

Take $\epsilon \in(0, 1)$ such that $1+1/q-1/p_1-1/p_2=1+1/q-1/p > \epsilon$, and let $\nu= 1/(1-\epsilon)$.
Then
\begin{align*}
\big| \fint_{B_{r}}(b(x)-b_{B_{2 s}})T (f_1,f_2)(x)\chi_{(B_{2 s})^\complement }(x)\big|
&\leq\Big( \fint_{B_{r}}|b(x)-b_{B_{2}}|^{v'}\Big)^{1 / v'}
\\&\qquad\times\Big( \fint_{B_{r}}|T (f_1,f_2)(x)\chi_{(B_{2 s})^\complement }(x)|^{v}\Big)^{1 / v}.
\end{align*}
From Lemma $\ref{lem2.2}$ and $(\ref{eq1.14})$, it follows that, for $r>4s\geq4$,
\begin{equation}\label{eq5.10}\,\
\begin{split}
\Big( \fint_{B_{r}}|b(x)-b_{B_{2 s}}|^{v'}dx\Big)^{1 / v'}
&\lesssim \int_{2 s}^{r} \frac{\psi(0, t)}{t} d t\|b\|_{\mathcal{L}^{(1, \psi)}}
\lesssim \psi(0, r) (\log r)\|b\|_{\mathcal{L}^{(1, \psi)}}
\\&\lesssim \varphi(0, r)^{1 / q-1 / p} (\log r)\|b\|_{\mathcal{L}^{(1, \psi)}}.
\end{split}
\end{equation}
By $(\ref{eq5.6})$ it follows that
\begin{equation}\label{eq5.11}\,\
\begin{split}
\Big(\int_{B_{r} \backslash  B_{s}}|T (f_1,f_2)(x)|^{v} d x\Big)^{1 / v}
&\lesssim\Big(\int_{B_{r} \backslash  B_{s}} \big(\frac{1}{|x|^{2n }}\|f_1\|_{L^{1}}\|f_2\|_{L^{1}}\big)^{\nu} d x\Big)^{1 / v}
\\&\lesssim \frac{1}{ |B_r| } \|f_1\|_{L^{1}}\|f_2\|_{L^{1}}.
\end{split}
\end{equation}
By $(\ref{eq5.10})$ and $(\ref{eq5.11})$ we have
\begin{equation*}
\begin{split}
\begin{aligned}
\Big| \fint_{B_{r}}\left(b-b_{B_{2 s}}\right)T (f_1,f_2)\left(\chi_{(B_{2 s})^\complement}\right)\Big|
&\lesssim \varphi(0, r)^{1 / p} (\log r) \frac{1}{r^{n / v+n}} \|b\|_{\mathcal{L}^{(1, \psi)}} \| f _1\|_{L^{1}} \| f_2 \|_{L^{1}}
\\&=\frac{\log r}{r^{n(2+1 / q-1 / p-\epsilon)}}\left(\frac{1}{r^{n} \varphi(0, r)}\right)^{1 / p-1 / q}\|b\|_{\mathcal{L}^{(1, \psi)}}
\\&\qquad\times\|f_1\|_{L^{1}}\|f_2\|_{L^{1}}
\rightarrow 0,\, as\,r \rightarrow \infty,
\end{aligned}
\end{split}
\end{equation*}
since $r^n\varphi(0,r)$ is almost increasing and
$2+1/q-1/p -\epsilon>1,\,1/p-1/q>0$.
This prove $(\ref{eq5.9})$ and completes the proof of Case 1.

{\bf Case 2.} We will show (\ref{eq5.4}) for general $f_1 \in L^{(p_1,\varphi_1)}(\mathbb{R}^n)$, and $f_2 \in L^{(p_2,\varphi_2)}(\mathbb{R}^n)$ with compact support. Fixing $r>0$ and $\supp f_2\subset B_{2r}$, we decompose $f_1=f_1\chi _{B_{2r}}+f_1\chi _{(B_{2r})^\complement }$, then
\begin{align*}
[b,T](f_1,f_2)(x)&=[b,T](f_1\chi _{B_{2r}},f_2)(x)+[b,T](f_1\chi _{(B_{2r})^\complement },f_2)(x)
\\&=[b,T](f_1\chi _{B_{2r}},f_2\chi _{B_{2r}})(x)+[b,T](f_1\chi _{(B_{2r})^\complement },f_2\chi _{B_{2r}})(x).
\end{align*}
To obtain $(\ref{eq5.4})$ it suffices to prove
$$[b,T](f_1\chi _{B_{2r}},f_2\chi _{B_{2r}})(x)\rightarrow 0~~and ~~[b,T](f_1\chi _{(B_{2r})^\complement },f_2\chi _{B_{2r}})(x)\rightarrow 0,\,as\,r\rightarrow\infty.$$

For $[b,T](f_1\chi _{B_{2r}},f_2\chi _{B_{2r}})(x)$, similar to  Case 1.

For $[b,T](f_1\chi _{(B_{2r})^\complement},f_2\chi _{B_{2r}})(x)$, using estimate of Lemma$~\ref{lem2.3}$ and Lemma$~\ref{lm3.2}$, $(\ref{eq1.14})$, we have
\begin{align*}
[b,T](f_1\chi _{(B_{2r})^\complement},f_2\chi _{B_{2r}})(x)&\lesssim\int_{r}^{\infty }\frac{\psi(z,t) }{t}   \Big( \int_{t}^{\infty }\frac{\varphi (z,u)^{1/p} }{u}du  \Big)dt \| b  \|_{\mathcal{L}^{(1,\psi )} } \prod_{i=1}^{2}  \| f_i \|_{L^{(p_i,\varphi _i)}}
\\&\lesssim\varphi(z, r)^{1 / q}\|b\|_{\mathcal{L}^{(1, \psi)} }\left \| f_1 \right \|_{L ^{(p_1, \varphi_1)}}\left \| f_2 \right \|_{L ^{(p_2, \varphi_2)}}
\\&\lesssim\varphi(0, r)^{1 / q}\|b\|_{\mathcal{L}^{(1, \psi)} }\left \| f_1 \right \|_{L ^{(p_1, \varphi_1)}}\left \| f_2 \right \|_{L ^{(p_2, \varphi_2)}}.
\end{align*}
Then
\begin{equation*}
\begin{split}
\begin{aligned}
\fint_{B_{r}}[b,T](f_1\chi _{(B_{2r})^\complement},f_2\chi _{B_{2r}})
\lesssim& \varphi(0, r)^{1 / q}\|b\|_{\mathcal{L}^{(1, \psi)} }\left \| f_1 \right \|_{L ^{(p_1, \varphi_1)}}\left \| f_2 \right \|_{L ^{(p_2, \varphi_2)}}\rightarrow 0,~~as~~r\rightarrow\infty.
\end{aligned}
\end{split}
\end{equation*}
Therefore, we have $(\ref{eq5.4})$ for all general
$f_1 \in L^{(p_1,\varphi_1)}(\mathbb{R}^n)$, and $f_2 \in L^{(p_2,\varphi_2)}(\mathbb{R}^n)$ with compact support.

{\bf Case 3.} We will show (\ref{eq5.4}) for general $f_1 \in L^{(p_1,\varphi_1)}(\mathbb{R}^n)$ and $f_2 \in L^{(p_2,\varphi_2)}(\mathbb{R}^n)$. Fixing $r>0$, we decompose $f_i=f_i\chi _{B_{2r}}+f_i\chi _{(B_{2r})^\complement },\,i=1,2$, we have
\begin{align*}
T(f_1,f_2)(x)&\leq T(f_1\chi _{B_{2r}},f_2\chi _{B_{2r}})(x)+T(f_1\chi _{(B_{2r})^\complement },f_2\chi _{B_{2r}})(x)
\\&\qquad+T(f_1\chi _{B_{2r}},f_2\chi _{(B_{2r})^\complement })(x)+T(f_1\chi _{(B_{2r})^\complement },f_2\chi _{(B_{2r})^\complement })(x)
\\&:=\mathrm{I}(x)+\mathrm{II}(x)+\mathrm{III}(x)+\mathrm{IV}(x).
\end{align*}
By Case 1 and 2, we get
$$
\fint_{B_r} \mathrm{I}(x)\rightarrow 0,\,\fint_{B_r} \mathrm{II}(x)\rightarrow 0,\,and \, \fint_{B_r}\mathrm{III}(x)\rightarrow 0, \,as\,r\rightarrow\infty.
$$
 and using Lemma $\ref{lm3.4}$ we obtain $\fint_{B_r}\mathrm{IV}(x)\rightarrow 0,\,as\,r\rightarrow\infty$. Therefore, we have  (\ref{eq5.4}) for all $f_1\in L^{(p_1, \varphi_1)}(\mathbb{R}^n)$, and $f_2\in L^{(p_2, \varphi_2)}(\mathbb{R}^n)$. The proof of Theorem$~\ref{the1.6}$(i) is completed.
\end{proof}
\begin{proof}[\bf Proof of Theorem$~\ref{the1.6}$(ii)]
We use the method by Janson \cite{JSS}. Since $1/{K(\vec{z})}$ is infinitely differentiable in an open set, we may choose $z_0\in \mathbb{R}^n$, such that $z_0\neq 0$ and $\delta>0$ such that $1/{K(\vec{z})}$ can be expressed in the neighborhood $\mathcal{B}=B((z_0,\ldots,0),2\sqrt{m}\delta)$\\$\subset (\mathbb{R}^{n})^m$  as an absolutely convergent Fourier series of the form
$$\frac{1}{K(\vec{y})} =\sum a_je^{i\left \langle \vec{v_k} ,\vec{y}  \right \rangle },$$
where $\sum a_j<\infty$ and the vectors $\vec{v_k} \in  (\mathbb{R}^{n})^m$ is irrelevant, but we
will at times express them as $\vec{v_k} = (v_k^1
,\ldots,v_k^m)$ and $\vec{y}=(y_1,\ldots,y_m)\in (\mathbb{R}^{n})^m$.\par
Let $z_1=\delta^{-1}z_0$. If $| z-z_1 |<2\sqrt{m}$, then
\begin{align*}
&(| y_1-z_1|^2+| y_2|^2+\dots+ | y_m|^2 )^{1/2}<2\sqrt{m}
\\&\Rightarrow(| y_1-\frac{z_0}{\delta}|^2+| y_2|^2+\dots +| y_m|^2 )^{1/2}<2\sqrt{m}
\\&\Rightarrow(| \delta y_1-z_0|^2+| \delta y_2|^2+\dots +| \delta y_m|^2 )^{1/2}<2\sqrt{m}\delta,
\end{align*}
we obtain
\begin{equation}\label{eq5.14}
\frac{1}{K(\vec{y})} =\frac{\delta^{-mn}}{K(\delta y_1,\ldots,\delta y_m)} =\sum a_j\delta^{-mn}e^{i\delta \langle \vec{v_k},\vec{y}  \rangle }.
\end{equation}

Let $B_0=B(x_0,r)\subset\mathbb{R}^{n}$, set $\widetilde{z}=x_0-rz_1,\, B'=B(\widetilde{z},r)\subset\mathbb{R}^{n}$. Then, for any $x\in B_0$ and $y_i\in B',\,i=1,\ldots,m$, which in turn implies
$$\Big| \frac{x-y_i}{r} -z_1\Big| =\Big| \frac{x-y_i}{r} -\frac{x_0-\widetilde{z} }{r}\Big|
=\Big| \frac{x-x_0}{r} -\frac{y_i-\widetilde{z} }{r}\Big|\le \Big| \frac{x-x_0}{r} \Big|+\Big| \frac{y_i-\widetilde{z} }{r} \Big|\le 2,$$
and
$$\Big| \frac{y_i-y_j }{r} \Big|\le 2,\,i\neq j,$$
which implies
$$\Big ( | \frac{x-y_i}{r} -z_1|^2+\sum_{j\neq i} | \frac{y_i-y_j }{r} |^2 \Big )^{1/2} \le 2\sqrt{m}.$$
Hence, we conclude that
$$K\Big(\frac{y_{i}-x_{i}}{r}, \frac{y_{i}-y_{1}}{r}, \ldots, \frac{y_{i}-y_{i-1}}{r}, \frac{y_{i}-y_{i+1}}{r}, \ldots, \frac{y_{i}-y_{m}}{r}\Big),$$
can be expressed as an absolutely convergent Fourier series as (\ref{eq5.14}) for all
$x \in B$ and $y_1,...,y_m \in B'$.
Since $$\sum_{j=1}^{m} \sum_{i \neq j}(b_{j}(y_{j})-b_{i}(y_{i}))=0,$$
and
$$\int_{(B')^m}\sum_{j=1}^{m} \sum_{i\neq j} (b_i(y_i)-b_j(y_i))d\vec{y} =|B'|^m\sum_{j=1}^{m} \sum_{i\neq j}(b_{B'}^i-b_{B'}^j)=0,$$
which implies that
\begin{align*}
&\int_{B^{m}} \Big|\sum_{j=1}^{m}(b_{j}(x_{j})-(b_{j})_{B^{\prime}})\Big|d\vec{x}
\\&=\frac{1}{|B|^{m}} \int_{B^{m}}\Big|\int_{(B^{\prime})^{m}} \sum_{j=1}^{m}(b_{j}(x_{j})-b_{j}(y_{j})) d \vec{y}\Big| d \vec{x}
\\&=\frac{1}{|B|^{m}} \int_{B^{m}}\Big|\int_{(B^{\prime})^{m}} \sum_{j=1}^{m}((b_{j}(x_{j})-b_{j}(y_{j}))+\sum_{i \neq j}(b_{i}(y_{i})-b_{j}(y_{j}))) d \vec{y}\Big| d \vec{x}
\\&=\frac{1}{|B|^{m}} \int_{B^{m}}\Big|\int_{(B^{\prime})^{m}} \sum_{j=1}^{m}((b_{j}(x_{j})-b_{j}(y_{j}))+\sum_{i \neq j}(b_{i}(y_{i})-b_{i}(y_{j}))) d \vec{y}\Big| d \vec{x}
\\& \leq \sum_{j=1}^{m} \frac{1}{|B|} \int_{B}\Big|\int_{(B^{\prime})^{m}}((b_{j}(x_{j})-b_{j}(y_{j}))+\sum_{i \neq j}(b_{i}(y_{i})-b_{i}(y_{j}))) d \vec{y}\Big| d x_{j}
\\&=\sum_{j=1}^{m} \frac{1}{|B|} \int_{B} \int_{(B^{\prime})^{m}}((b_{j}(y_{j})-b_{j}(x_{j}))+\sum_{i \neq j}(b_{i}(y_{j})-b_{i}(y_{i}))) d \vec{y} \cdot s_{j}(x_{j}) d x_{j},
\end{align*}
where
$$s_{i}(x_{i})=\operatorname{sgn}\Big\{ \int_{(B^{\prime})^{m}}\big [ (b_{i}(y_{i})-b_{i}(x_{i}))+\sum_{j \neq i}(b_{j}(y_{i})-b_{j}(y_{j})) \big ]d \vec{y} \Big \}, $$
For any $j \in \{1, 2,\ldots,m\}$, we denote
\begin{align*}
&f_{1}^{j, k}\left(x_{j}\right)=e^{-i \frac{\delta}{r} \nu_{k}^{1} \cdot x_{j}} s_{j}\left(x_{j}\right) \chi_{B}\left(x_{j}\right) ,
\\&f_{2}^{j, k}\left(y_{1}\right)=e^{-i \frac{\delta}{r} \nu_{k}^{2} \cdot y_{1}} \chi_{B^{\prime}}\left(y_{1}\right),
\\&\quad \quad\quad \quad\quad\vdots
\\&f_{j}^{j, k}\left(y_{j-1}\right)=e^{-i \frac{\delta}{r} \nu_{k}^{j} \cdot y_{j-1}} \chi_{B^{\prime}}\left(y_{j-1}\right) ,
\\&f_{j+1}^{j, k}\left(y_{j+1}\right)=e^{-i \frac{\delta}{r} \nu_{k}^{j+1} \cdot y_{j+1}} \chi_{B^{\prime}}\left(y_{j+1}\right) ,
\\&\quad \quad\quad \quad\quad\vdots
\\&f_{m}^{j, k}\left(y_{m}\right)=e^{-i \frac{\delta}{r} \nu_{k}^{m} \cdot y_{m}} \chi_{B^{\prime}}\left(y_{m}\right),
\end{align*}
and
$$g^{i, k}(y_{i})=e^{i \frac{\delta}{r} y_{i}\vec{\nu_{k}} } \chi_{B^{\prime}}(y_{i}),$$
 Set $C=\delta^{-m n}| B(0,1)|^{-m}$. Then
\begin{align*}
&\int_{B^{m}} \Big|\sum_{i=1}^{m}(b_{i}(x_{i})-(b_{i})_{B^{\prime}})\Big| d \vec{x}
\\&= \delta^{-m n} r^{m n} \sum_{i=1}^{m} \frac{\sum_{k} a_{k}}{|B|}\int_{(\mathbb{R}^{n})^{m+1}}((b_{i}(y_{i})-b_{i}(x_{i}))+\sum_{j \neq i}(b_{j}(y_{i})-b_{j}(y_{j})))
\\&\qquad\times K(y_{i}-x_{i}, y_{i}-y_{1},\cdots, y_{i}-y_{i-1}, y_{i}-y_{i+1}, \ldots, y_{i}-y_{m})
\\&\qquad\times f_{1}^{i, k}(x_{i}) f_{2}^{i, k}(y_{1}) \cdots f_{i}^{i, k}(y_{i-1}) f_{i+1}^{i, k}(y_{i+1}) \cdots f_{m}^{i, k}(y_{m}) g^{i, k}(y_{i}) d \vec{y} d x_{i}
\\&\leq C\sum_{i=1}^{m} \sum_{k}|a_{k}| |B|^{m-1} \int_{\mathbb{R}^{n}}|[ \vec{b}, T](f_{1}^{i, k}, \ldots, f_{m}^{i, k})(y_{i})||g^{i, k}(y_{i})| d y_{i}
\\&\leq C \sum_{i=1}^{m} \sum_{k}|a_{k}| |B|^{m} \varphi(B)^{1/q}\|[ \vec{b}, T](f_{1}^{i, k}, \ldots, f_{m}^{i, k})\|_{L^{(q,\varphi)}}
\\&\leq  C|B|^{m}\varphi(B)^{1/q}\|[\Sigma \vec{b}, T]\|_{L^{(p_{1},\varphi_1)} \times \cdots \times L^{(p_{m},\varphi_m)} \rightarrow L^{(q,\varphi)}} \sum_{k}|a_{k}|\prod_{j=1}^{m} \parallel f_j^{i,k}\parallel_{L^{(p_j,\varphi_j)}}.
\end{align*}
Since $\varphi_j$ is in $\mathcal{G}^{dec}$ and satisfies (\ref{eq1.10}), (\ref{eq1.12}), then  $\| f_j^{i,k}\|_{L^{(p_j,\varphi_j)}}=\| \chi_{B'}\|_{L^{(p_j,\varphi_j)}}\sim\frac{1}{\varphi(B')^{1/p_j}}$. Note that$~\varphi_j(B')\sim\varphi_j(B)$, since $| x_0-y_0 |=r| z_1| $. Then by (\ref{eq1.13}), we get
$\prod_{j=1}^{m} \| f_j^{i,k}\|_{L^{(p_j,\varphi_j)}}\lesssim\prod_{j=1}^{m}{\varphi_j(B)^{-1/p_j}}=\varphi(B)^{-1/p}$.

Consequently,
$$\int_{B^{m}} |\sum_{i=1}^{m}(b_{i}(x_{i})-(b_{i})_{B^{\prime}})| d \vec{x}
\lesssim \|[ \vec{b}, T]\|_{L^{(p_{1},\varphi_1)} \times \cdots \times L^{(p_{m},\varphi_m)} \rightarrow L^{(q,\varphi)}} |B|^m\varphi (B)^{1/q-1/p},$$
which implies that
\begin{align*}
\frac{1}{\psi (B)} \fint_{B^{m}}|\sum_{i=1}^{m}(b_{i}(x_{i})-b_B^{i})| d \vec{x}
& \leq \frac{2m}{\psi (B)} \fint_{B^{m}}|\sum_{i=1}^{m}(b_{i}(x_{i})-b_{B^{\prime}}^{i})| d \vec{x}
\\& \lesssim\|[\vec{b}, T]\|_{L^{(p_{1},\varphi_1)} \times \cdots \times L^{(p_{m},\varphi_m)} \rightarrow L^{(q,\varphi)}},
\end{align*}
where the last inequality follows from (\ref{eq1.16}). By Lemma $\ref{lm3.5}$, for all $j=1,\ldots,m$, we have $\| b_j\|_{\mathcal{L}^(1,\psi)}\lesssim\|[b_j, T]\|_{L^{(p_{1},\varphi_1)} \times \cdots \times L^{(p_{m},\varphi_m)} \rightarrow L^{(q,\varphi)}}$
and complete the proof of Theorem$~\ref{the1.6}$.
\end{proof}


\end{document}